\documentclass[12pt]{article}
\usepackage[utf8]{inputenc}
\usepackage{amsmath,amssymb}
\usepackage{times}
\usepackage{array}
\usepackage[english]{babel}
\usepackage{wrapfig} 
\usepackage{amssymb,amsthm,color}
\input epsf
\usepackage{graphicx}

\theoremstyle{definition}
\newtheorem{defn}{Definition}[section]
\newtheorem{lemma}{Lemma}[section]
\newtheorem{proposition}{Proposition}[section]
\newtheorem{cor}{Corollary}[section]
\newtheorem{theorem}{Theorem}[section]

\newtheorem{remark}{Remark}

\def\conv{{\hbox{conv}\,}}

\def\diam{\hbox{diam}\,}
\def\dist{\hbox{dist}\,}

\def\bN{\mathbb{N}}
\def\bR{\mathbb{R}}

\def\d{\partial}\def\p{\partial}

\def\ep{\epsilon}

\def\cC{{\mathcal{C}}}

\def\cI{{\mathcal{I}}}
\def\cK{{\mathcal{K}}}
\def\cM{{\mathcal{M}}}

\numberwithin{equation}{section}
\newcommand{\subjclass}[1]{\bigskip\noindent\emph{2010 Mathematics Subject Classification:}\enspace#1}
\newcommand{\keywords}[1]{\noindent\emph{Keywords:}\enspace#1}

\begin{document}
%opening
\title{The planar Least Gradient problem in convex domains}
\author{Piotr Rybka, Ahmad Sabra\\
%\address{
Faculty of Mathematics, Informatics and  Mechanics\\
The University of Warsaw\\
ul. Banacha 2, 02-097 Warsaw, POLAND}
%\today

\maketitle
\begin{abstract}
We study the two dimensional least gradient problem in a convex, but
not necessary strictly convex region. We look for solutions in the
space of $BV$ functions satisfying the boundary data $f$ in trace
sense. We  assume that $f$ is in $BV$ too. We state admissibility
conditions on the trace and on the domain that are sufficient for existence of solutions.

 
\end{abstract}

\section{Introduction}
We study  the least gradient problem
\begin{equation}\label{lg}
 \min \left\{ \int_\Omega |D u|: \ u\in BV(\Omega),\ T u = f\right\},
\end{equation}
where $\Omega$ is a bounded convex region in the plane  with Lipschitz boundary. We denote by $T: BV(\Omega) \to L^1(\partial\Omega)$ the trace operator.
We stress that we are interested only 
in solutions to (\ref{lg}) satisfying 
\begin{equation}\label{trace}
 T u = f,
\end{equation}
where $f$ is in $BV(\partial\Omega)\subsetneq  L^1(\partial\Omega)$.

Since publishing of the paper by Sternberg-Williams-Ziemer, \cite{sternberg}, the least gradient problem was broadly studied. In \cite{sternberg}  existence and uniqueness for continuous data $f$ were shown when the domain $\d\Omega$ has a non-negative mean curvature (in a weak sense) and $\d\Omega$ is not locally area minimizing.  These conditions in $\bR^2$ reduce to 
strict convexity of $\Omega$. 

The case of general $f\in L^1(\partial\Omega)$ was studied in
\cite{mazon}. However, in \cite{mazon} uniqueness is lost and
(\ref{trace}) does not hold in the classical sense. Moreover, the
authors of \cite{spradlin} show that the space of traces of solutions
to the least gradient problem is essentially smaller than
$L^1(\partial\Omega)$, then a solution to (\ref{lg}) does not
necessarily exist for $L^1$ data. Any characterization of the space of
traces of least gradient functions does not seem to be known
yet. However,  $f\in BV(\d\Omega)$ is sufficient for existence when
$\Omega\subset \bR^2$ is strictly convex, see \cite{gorny}. Another
approach to existence is presented in \cite[Theorem 1.5]{moradifam}. 
It is not applicable here, because the barrier condition does not hold.

A motivation to study (\ref{lg}) comes from the conductivity problem and free material design, see \cite{grs}. A weighted least gradient problem appears in medical imaging, \cite{nachman}. This is also the motivation to investigate the anisotropic version of (\ref{lg}), see \cite{jerrard}. 

A geometric problem of least area leads to a minimization problem like  (\ref{lg}),  where $\int_\Omega |Du|$  for $u\in BV(\Omega)$ is replaced by  $\int_\Omega f(D u)$,
%$$\inf \left\{ \int_\Omega f(D u): \ u\in BV(\Omega), T u = f\right\},$$
where $f$ has linear growth, but need not be 
1-homogeneous. This problem attracted attention in the seventies', see \cite{soucek}, and it is active today \cite{beck}.

A common geometric restriction on the domain $\Omega$ in the above
papers boils down to strict convexity in plane domain. 
%, (this condition has been relaxed in multidimensional case).
Here, the main difficulty is the lack of strict convexity of a convex region $\Omega$. Since we do not expect existence of solutions to (\ref{lg}), even in the case of continuous $f$, if $\Omega$ is merely convex, we
%for discontinuous data even in the case of strictly convex $\Omega$, we 
have to develop a proper tool. For this purpose we state admissibility
conditions on the behavior of $f$ on flat parts of the boundary. We
first do this when the data are continuous, in this case the
admissibility condition \# 1, means monotonicity of $f$ on any flat
part $\ell$. Condition \# 2 is geometric, strictly related to the
given boundary data $f$. Namely, it means that the data on a flat
boundary may attain maxima or minima on large sets, making creation of
level sets of positive Lebesgue measure advantageous,  see
\S\ref{three}. Since we deal here with oscillatory data, we assume
that $f$ is not only continuous but also has a bounded total variation. This additional assumption is used in our analysis, but we do not know if this condition is necessary. 
These admissibility conditions have to be modified in the case of discontinuous data, see \S\ref{subsec:disc}.

In Section \ref{adm}, we will explain the notion of the flat part of the boundary of $\Omega$, denoted by $\ell$. In the same section, we introduce the admissibility conditions for continuous and discontinuous functions.
Once we have them, we can state our main results. We make an additional assumption, when the number of flat parts is infinite. Namely, we assume that they have a single accumulation point.

We will address first the case of continuous data, which is interesting for its own sake.
\begin{theorem}\label{tsuf-C}
Let us suppose that $f\in C(\d\Omega)$, %\cap BV(\d\Omega)$, 
$\Omega$ is an open, bounded and convex set, $\{\ell_\alpha\}_{\alpha\in \cI}$ is the family of flat parts of $\d\Omega$. 
The flat parts have at most one accumulation point and if so the
condition from Definition \ref{one-side} holds. If $f$ satisfies the
admissibility conditions (\ref{admC1}) or  (\ref{admC2}) on all flat
parts of $\d\Omega$, then problem (\ref{lg}) has exactly  one
solution, i.e. the boundary data are attained in  the trace sense.
\end{theorem}

%\begin{remark}
 %We need the additional assumption of $f$ being an element of $BV$ when we deal with the situations of infinitely many oscillations or infinitely many flat parts.
%\end{remark}

The strategy of our proof is as follows, we construct a sequence of strictly convex regions $\Omega_n$ converging to $\Omega$ in the Hausdorff distance. We also provide approximating data on $\d\Omega_n$. By classical result, see \cite{sternberg}, we obtain a sequence of continuous solutions $v_n$ to the least gradient problem (\ref{lg}) on $\Omega_n$.  After estimating the common modulus of continuity, we may pass to the limit using a result by \cite{miranda}. This is done in Section \ref{sCdata}.

Next, we extend this result to $f\in BV$. %We will make another simplification, justified by the approximation technique we use, namely on strictly convex parts of $\d\Omega$ the given $f$ will be continuous.  
%In order to make the presentation clear, we make a number of simplifications. The first one, which is not essential for the analysis is the following. We consider the boundary conditions $f$ such that they are not only in $BV(\d\Omega)$ but also they are continuous away from the flat parts of the boundary.\textcolor{red}{??}
%i.e. line intervals, which are contained in $\d\Omega$, this will be explained below.
We admit an infinite  number of such flat pieces of the boundary.
However, for simplicity, we assume that they may accumulate at just one point. It turns out that the presence of infinitely many  flat parts leads to additional difficulties, which are more pronounced when $f$ is discontinuous. Our final result is:
\begin{theorem}\label{tsuf-D}
Let us suppose that the geometric assumptions on $\Omega$ specified in Theorem \ref{main1} hold, in particular $\Omega$ is convex. Moreover,
$f\in BV(\d\Omega)$ and it satisfies 
the admissibility conditions \eqref{admDC} or \eqref{admDC2}. Then, there exists a solution $u$ to the least gradient problem (\ref{lg}).
 \end{theorem} 
 
We establish Theorem \ref{tsuf-D} by a proper approximation of
discontinuous data by continuous ones. We approximate $f$ from above
by a monotone decreasing sequence $h_n$ and from below by a monotone
increasing sequence $g_n$, this is done in Corollary \ref{Aprox2}. 
By Theorem \ref{tsuf-C} we will have sequences of solutions to
(\ref{lg}), $u_n$ corresponding to $h_n$ and $v_n$ corresponding to
$g_n$. 

The comparison principle will imply monotonicity of $u_n$ and $v_n$, hence $u_n$ (resp. $v_n$) will converge pointwise convergence to $u$ (resp. $v$) and we will have $u\ge v$. 
We will have tools to prove that the limit functions $u$ and $v$ have the correct trace. In general $u$ and $v$ need not be equal.

It is a natural and  interesting question, what happens when the
admissibility conditions are violated. This problem requires quite
different technique and it will be addressed elsewhere, here we present  simple examples in Section \ref{examples}. We claim that our admissibility conditions are sufficient and (almost) necessary for existence. 

It is well-known that for discontinuous data the uniqueness of
solutions is lost, see \cite{mazon} for a proper version of Brothers
example. However, a recent article \cite{gorny2} provides a classification of multiple solutions. This result is valid for strictly convex as well as for merely convex regions. We do not present further details in this direction.

\section{Admissibility criteria}\label{adm}

It is important to monitor the behavior of data on flat parts of the boundary. We need to introduce a few pieces of notation for this purpose. 

\begin{defn} \label{plask}
 A non-degenerate line segment $\ell$ will be called a {\it flat part} of the boundary of $\Omega$ if $\ell \subset\d\Omega$ and $\ell$ is maximal with this property. In particular, this definition implies that $\ell$ is closed. The collection of all flat boundary parts is denoted by $\{\ell_k\}_{k\in \cK}$.
\end{defn}

We expect that data must satisfy additional conditions in order to ensure existence of solutions. We will state these admissibility conditions for continuous and discontinuous data separately.

\subsection{The case of continuous data}\label{three}
Through  this  subsection we consider only continuous  boundary values $f$.

\begin{defn} \label{admC1} We shall say that
a continuous function $f\in C(\d\Omega)$ {\it satisfies the admissibility condition \#1} on a flat part $\ell$ if and only if $f$ restricted to $\ell$ is monotone.
\end{defn}

In order to present the admissibility conditions for functions which are not monotone we need more auxiliary notions. 
We associate with $f$ on a flat piece of the boundary, $\ell$, a family of intervals $\{I_i\}_{i\in \cI}$ such that $\bar I_i =[a_i, b_i]\subset \ell$. On each $I_i$  function  $f$ is constant and attains a local maximum or minimum and each  $I_i$ is maximal with this property. We note that in case of continuous functions $f$ the intervals $I_i$ are automatically closed. We also set $e_i = f(I_i^o)$, $i\in\cI$.
This definition of $e_i$ will be good also in the case of discontinuous data. For the sake of making the notation concise, we will call $I_i$ a {\it hump}.

After this preparation, we state the admissibility condition for non-monotone functions.
\begin{defn} \label{admC2}
 A continuous function $f$, which is not monotone on a flat part $\ell$, {\it satisfies the admissibility condition \#2} if and only if for each hump $I_i = [a_i, b_i]\subset\ell$, $i\in\cI$ the following inequality holds, 
 \begin{equation}\label{A}
 \dist(a_i, f^{-1}(e_i)\cap(\d\Omega\setminus I_i)) + \dist(b_i, f^{-1}(e_i)\cap(\d\Omega\setminus I_i)) < |a_i-b_i|.
\end{equation}
In addition we require that if $y_i$, $z_i\in \d\Omega$ are such that
\begin{equation}\label{defyz}
\dist(a_i, f^{-1}(e_i)\cap(\partial\Omega\setminus I_i)) = 
\dist(a_i, y_i),\qquad 
\dist(b_i, f^{-1}(e_i)\cap(\partial\Omega\setminus I_i)) = \dist(b_i, z_i),
\end{equation}
then $y_i,$ $z_i\in \d\Omega \setminus \ell$.
\end{defn}
We use here the notation $\dist (x,\emptyset)=+\infty$. Obviously, the admissibility condition \#2 does not hold if $f$ has a strict local maximum or minimum.

\begin{remark}
By definition, $f$ attains local maximum/minimum on $I_i$, even if a the hump contains any endpoint of $\ell$. Later, we will make comments, see Remark \ref{rem2}, about intervals contained in $\ell$, where $f$ is constant and such that one of their endpoints is a point from $\d\ell$, which are not hump in our sense.
\end{remark}

We would like to discover what are the consequences of the admissibility conditions.
In particular, we would like to know if the restriction of $f$ to a flat part 
$\ell$ can have an infinite number of local minima or maxima. Interestingly, the
answer depends upon the geometry of $\Omega$. Namely, we can prove the following
statement.
\begin{proposition}\label{pr-2.1}
Let us suppose $\ell$ is a flat piece of boundary of $\Omega$. In addition, $\ell$
makes an obtuse angle with the rest of $\partial\Omega$ at its endpoints,
$\partial \ell =\{p_l, p_r\}$. Then, if $f$ satisfies on $\ell$ the admissibility
condition \#2, then $f|_\ell$ has a finite number of humps.
\end{proposition}

\begin{proof}
Let us introduce a strip $S(\ell)$, defined as
$$
S(\ell) = (\bigcup_{x\in\ell} L_x)\cap \Omega,
$$
where $L_x$ is the line perpendicular to $\ell$ and passing through $x$. We notice that $S(\ell)$ is a stripe, so the intersection of its boundary with a convex set,
$\partial S(\ell) \cap \Omega$, consists of two line segments $s_l$ and $s_r$.

We can have the following situations, each of $[a_i, y_i], [b_i, z_i] $ may either be contained in $S(\ell)$  
or each of 
$[a_i, y_i], [b_i, z_i]$ may intersect $s_l \cup s_r$. Here $y_i$'s and $z_i$'s are defined in (\ref{defyz}).

We claim that there is only a finite number of segments $[a_i, y_i], [b_i, z_i] $ 
contained in $S(\ell)$. Indeed, if it were otherwise, then $b_{k_n} - a_{k_n}$ would converge
to  zero for a subsequence $k_n$ converging to infinity. On the other hand
$\min\{ \dist(a_i, \partial \Omega\cap S(\ell)\setminus\ell), \dist(b_i, \partial \Omega\cap S(\ell)\setminus\ell)\}\ge c_0>0.$
But these two conditions combined contradict the admissibility conditions \#2.

We claim that there is only a finite number of segments $[a_i, y_i], [b_i, z_i] $
intersecting $s_l \cup s_r =  \partial S(\ell) \cap \Omega $. 
We have to consider a few cases. 
Let us suppose that $[a_i, y_i]\cap s_t \neq \emptyset$, where $t= l$ or $r$. Here we adopt the following convention, $\ell =[p_l, p_r]$, and 
$$
\dist(a_i, p_r) > \dist(a_i, p_l)\quad\hbox{and}\quad 
\dist(b_i, p_r) < \dist(b_i, p_l).
$$
If we keep this in mind and $t = r$, then the geometry automatically implies that $[b_i, z_i]\cap s_l \neq \emptyset$.

We notice that neither segment $[p_l, a_i]$ nor $[a_i, p_r]$ can contain  $I_k$. If  it did, then segments $[a_i, y_i]$ and $[b_i, z_i]$ would intersect $s_r \cup s_l$. If we take into account that the angle at $p$ is obtuse we see that the length of $[b_i, z_i]$ is bigger than $|b_i - a_i|$, but this is impossible. A similar argument works for $p = p_r$. Hence, our claim follows.
\end{proof}

Actually, we will make the above statement even  more precise.

\begin{proposition}\label{p-2.2}
 Let us consider $s_l \cup s_r =\partial S( \ell)\cap \Omega$ and the set of all humps contained in $\ell$, $\{I_i\}_{i\in \cI}$, where $\bar I_i = [a_i,b_i]$. If $\partial\Omega$ forms obtuse angles at $\partial \ell$ (or there are lines tangent to $\partial\Omega$ there), then 
each of the intersections 
$$
s_l \cap \left( \bigcup_{i\in \cI}([a_i,y_i]\cup [b_i,z_i])\right), \qquad
s_r \cap \left( \bigcup_{i\in \cI}([a_i,y_i]\cup [b_i,z_i])\right),
$$ 
%or $ s_r$ with intervals of the form $[a_i, y_i]$, $[b_i,z_i]$, 
consists of at most one point. Here, $y_i,z_i$, $i\in \cI$ satisfy (\ref{defyz}).  In particular, $\ell$ contains at most one hump.
\end{proposition}
\begin{proof}
Our notational convention is that 
$$
\dist(p_l, a_i)<\dist(p_r, a_i), \qquad \dist(p_r, b_i)< \dist(p_l, b_i).
$$

Let us suppose our claim does not hold. The geometry implies that if $[b_i,z_i]\cap s_l\neq \emptyset$, (resp. $[a_i,y_i]\cap s_r\neq \emptyset$), then automatically
$[a_i,y_i]\cap s_l\neq \emptyset$, (resp. $[b_i,z_i]\cap s_p\neq \emptyset$). If so, then the interval  $[b_i,z_i]$ (resp. $[a_i,y_i]$) is longer then $|b_i-a_i|$ because the angle at $p$ is obtuse. 
 
Let us suppose now that $[a_i,y_i]\cap s_t\neq \emptyset$ and   $[a_j,y_j]\cap s_t\neq \emptyset$, where $t=l$ or $t=r$ and $j\neq i$. If this happens, we can find a $b_k$ in the interval $[a_i, a_j]\subset\ell$. We  can use the part that we have already shown to deduce our claim.

The same argument applies when $[b_i,z_i]\cap s_t\neq \emptyset$ and   $[b_j,z_j]\cap s_t\neq \emptyset$, where $t=l$ or $t=r$ and $j\neq i$.

If $\ell$ contained more than one hump, then there existed
$[a_i,y_i]\cap s_l\neq \emptyset$ and $[a_j,y_j]\cap s_l\neq \emptyset$ or
$[b_i,z_i]\cap s_r\neq \emptyset$ and $[b_j,z_j]\cap s_r\neq \emptyset$ for $i\neq j$. We have already seen that this is impossible.
\end{proof}

We will make further  observations about the structure of admissibility conditions. A particularly interesting is the case when $\partial\Omega$ has an infinite number of flat parts. For the sake of simplicity of presentation, we will assume that $\{\ell_k\}_{k \in \cK}$ has a single accumulation point, i.e. if $\ell_k = [p^k_l, p^k_r]$, then $p^k_l, p^k_r \to p_0$, as $k$ goes to infinity. We assume that $\ell_k= [p^k_l, p^k_r]$ are so arranged that $\dist( p^k_l,p_0) > \dist( p^k_r,p_0)$. In order to avoid  unnecessary complications, we assume that `almost all $\ell_k$'s are on one side of $p_0$'. In order to make it precise, we will denote by $H(\ell)$ the half-plane, whose boundary contains a line segment (or a line) $\ell$, $H(\ell)\cap \partial\Omega\neq \emptyset$ and $\Omega \subset H(\ell)$. 

\begin{defn}\label{one-side}
 Let $\nu$ be a unit vector parallel to $\ell$, by saying that {\it almost all $\ell_k$'s are on one side of $p_0$} we mean that there is a choice of $\nu$ so that infinitely many $\ell_k$ are contained in $\{y\in \bR^2:\ (y-p_0)\cdot\nu 
<0\}$, while the half-plane $\{y\in \bR^2:\ (y-p_0)\cdot\nu >0\}$ contains only their finite number.
\end{defn}
 
After this preparation, we will see what of restrictions imposes the admissibility condition \#2 on the boundary data and on any infinite sequence of flat parts. The boundary $\partial\Omega$ at $p_0$ may have a tangent line or form an angle. The angle may be obtuse (including the case of a tangent line) or acute. We see that solutions depend on the measure of the angle. Namely,
we noticed that if the angle is obtuse, then we can only have single humps $I_k$ on flat parts and a finite number of humps on $\ell$. If the angle is acute, then we may have an infinite number of humps on $\ell$,  accumulating at $p_0$. In addition, there may be an infinite number of flat parts accumulating at $p_0$.

In this subsection, we will consider the case of an obtuse angle. The acute angle will be treated later, in Lemma \ref{l-inter}, when $p_0$ is an endpoint of a flat part and $f$ restricted to $\ell$ satisfies the admissibility condition \#2.

When
we deal with  a sequence of flat parts $\{\ell_k\}_{k=1}^\infty$, then we 
assume that they are so numbered that $\dist(\ell_k,p_0)\to 0$ and almost all $\ell_k$ are on one side of $p_0$. We are ready to present the following fact.

\begin{lemma}\label{lm2.0}
Let us assume that the above condition holds and $\partial\Omega$ forms an obtuse angle at $p_0$. Moreover,  $f\in C(\partial\Omega)$ satisfies the admissibility conditions \#1 and \#2.
Then, there is $\rho>0$ with the following property.
%The intersection $B(p_0, \rho)\cap\partial\Omega$ is connected.
If $\ell_k \subset B(p_0, \rho)\cap\d\Omega$, then intervals $[a_k,y_k], [b_k,z_k]$, where $y_k, z_k$ satisfy (\ref{defyz}), must intersect  $\partial S(\ell_k) \cap \Omega$.
\end{lemma}

\begin{proof}
We note that due to Proposition \ref{p-2.2} only one hump $I_k$ may be contained in $\ell_k$. Due to the obtuse angle at $p_0$, the length of each component of $\partial S(\ell_k)$ may be made strictly bigger than a fixed number $c_0$, while the length of $\ell_k$ goes to zero. If our claim were not true than the lengths of  $[a_k,y_k], [b_k,z_k]$ would exceed  $c_0$ in violation of the admissibility condition \#2.
\end{proof}

\subsection{The case of discontinuous data}\label{subsec:disc}
If function $f$ is not continuous, then the admissibility conditions have to be adjusted. In particular,
%\#2 goes without changes, but 
this is true regarding condition \#1. This is done below.

\begin{defn}\label{admDC} Let us suppose that $f:\d\Omega\to \bR$ is a function of bounded variation and $f$ restricted to  $\ell =[p_l,p_r]$ is
monotone. We shall say that $f$ {\it satisfies the admissibility condition \#1} if 
one of the following conditions holds:\\
(i)  $f$ is continuous at both endpoints of $\partial\ell$;\\
(ii) there is $\ep>0$ such that function $f$ restricted to $\ell \cup (B(x_0,\ep)\cap \d\Omega)$ is monotone, where $x_0\in\d\ell$ and $f$  is continuous in $\d\ell\setminus \{x_0\}$;\\
(iii) there is $\ep>0$ such that function $f$ restricted to 
$\ell \cup (\d\Omega \cap (B(p_l,\ep)  \cup B(p_r,\ep))$
is monotone. 
\end{defn}

If we recall  the notation introduced before Definition \ref{admC2}, then we 
notice that the definition of $e_i$ as $e_i = f(I_i^o)$ is unambiguous also in the case of discontinuous $f$.

\begin{defn}\label{admDC2}
We say that a function $f$ belonging to $BV(\d\Omega)$ satisfies the admissibility condition \#2 on a flat part $\ell$ if and only if for each closure $[a_i, b_i]$ of a hump, $i\in\cI$, the following inequality holds,
 \begin{equation}\label{DA}
 \dist(a_i, f^{-1}(e_i)\cap(\d\Omega\setminus I_i)) + \dist(b_i, f^{-1}(e_i)\cap(\d\Omega\setminus I_i)) < |a_i-b_i|.
\end{equation}
In addition we require that if $y_i$, $z_i\in\d\Omega$ are such that
$$
\dist(a_i, f^{-1}(e_i)\cap(\partial\Omega\setminus I_i)) = 
\dist(a_i, y_i),\quad 
\dist(b_i, f^{-1}(e_i)\cap(\partial\Omega\setminus I_i)) = \dist(b_i, z_i),
$$
then $y_i$, $z_i\in\d\Omega\setminus\ell.$ 
\end{defn}

We notice the geometric observations made in previous a subsection, i.e. Proposition \ref{pr-2.1}, Lemma \ref{lm2.0}, Corollary \ref{p-2.2}, are valid for discontinuous data satisfying the admissibility conditions. The proof does not depend on continuity.

\section{Construction of solutions for continuous data}\label{sCdata}

Solutions to (\ref{lg}) are constructed by the same limiting process which we used in \cite{grs}. We first find a sequence of strictly convex domains, $\{\Omega_n\}_{n=1}^\infty$, approximating $\Omega$. Then, we define $f_n$ on $\d\Omega_n$ in a suitable way. After this preparation, we invoke a classical result, \cite{sternberg}, to conclude existence of $\{v_n\}_{n=1}^\infty$, solutions to the least gradient problem in $\Omega_n$. 

The construction of strictly convex region $\Omega_n$ is easy, when at both endpoints of a flat piece there is a corner, then we simply add a piece of a circle arc. But when there is a tangent at one of the endpoints of $\ell$, then we have to perform  additional reasoning. This construction is presented in the following Lemma.

\begin{lemma}\label{prap}
 Let  us suppose that $\Omega\subset\bR^2$ is a convex region with finitely many flat parts $\ell_k$, $k\in\cK=\{1,\ldots, K\}$. Then,
 there is a sequence of strictly convex bounded regions, $\Omega_n$ containing $\Omega$ and such that $\bar\Omega_n$  converges to
 $\bar\Omega$ in the Hausdorff metric. 
\end{lemma}
\begin{proof}
We define $\Omega_n$ as follows. Let us suppose that 
$$
V=\{ v_m: \ m=1,\ldots, \cM\}
$$ 
is the set of all endpoints of flat pieces $\ell_k$, $k\in \cK$, of
$\d\Omega$, such $\ell_k$ meets  $\d\Omega\setminus\ell_k$ at  a positive
  angle, i.e. $\ell_k$ and the rest of $\d\Omega$ form a corner. Let us
  fix at each $v_j\in V$ a line, $L_j$, passing through $v_j$, which otherwise
  does not intersect $\Omega$.  We will construct $\Omega_n$ by adding
  a set bounded by $\ell_k$ and an arc with the same end-points as $\ell_k$.

1) First, we consider such $\ell_k$, that $\d \ell_k\subset V$. In this case,
we take a ball $B(p^n_k, r^n_k)$ with the following properties. %such that $p^n_k\in \bR^2$. %Omega$. 
The flat piece
$\ell_k$ is a cord of $S^1(p^n_k, r^n_k)$, the maximal distance of any
point from the arc $c^n_k = S^1(p^n_k, r^n_k)\setminus \Omega$ to $\ell_k$
does not exceed $\frac1n$ and $c^n_k$ intersects lines $L_j$ only at
$\d\ell_k$. Now, in order to define $\Omega_n$, we take the region
bounded by $c^n_k$ and $\ell_k$.

2)  Let us suppose that $p_1$ is an endpoint of $\ell_k$, which does not belong to $V$. We may assume that there is a positively oriented coordinate system, such that $L$ is the line containing $\ell_k $ and it coincides with the first coordinate axis and $\Omega$ is contained in the upper half plane. We may assume that $p_1 =(0,0)$ and $\ell_k =[-d,0]\times\{0\}$. There is $\rho>0$, such that 
$$
\d\Omega \cap B(p_1,\rho) =\{(x_1,x_2):\ x_2 =\psi^1(x_1), \ |x_1| <\bar\rho \}.
$$
By assumption on $p_1$, $\psi^1(0) =0$ and the derivative of $\psi^1$ exists at $x_1=0$ and 
$\frac d{dx_1} \psi^1(0)=0$. Since $\frac d{dx_1} \psi^1$ is increasing and there is no corner of $\d\Omega$ at $p_1$, we conclude that $\frac d{dx_1} \psi^1(x) \to 0$ as $x\to 0$. In particular, we can find positive $x_n$ converging to zero, such that $a_n:=\frac d{dx_1} \psi^1(x_n) \to 0$. We set 
$$
g_n(x) =\left\{
\begin{array}{ll}
 a_n(x-x_n) + \psi^1(x_n)& x\in [0,x_n],\\
\psi^1(x_n) - a_n x_n& x<0.
\end{array}
\right. 
$$
We define a convex function $\psi_n:(-\infty, \rho)\to \bR$ by the following formula,
$$
\psi^1_n(x) = \left\{
\begin{array}{ll} 
\psi^1(x) & x> x_n,\ (x,\phi(x))\in B(p_1,\rho),\\
 \frac12(\psi^1(x) + g_n(x))& x\in[0,x_n],\\
\frac12  (\psi^1(x_n) - a_n x_n )& -d<x<0.
\end{array}
\right. 
$$
We notice that $\psi^1_n$ is convex, has a corner at $x=0$, because the left-hand-side and the right-hand-side derivatives are different and $\psi^1_n$ converges uniformly to $\psi^1$.

2) We have to consider the other endpoint of $\ell_k$, i.e. $p_2:=(-d,0)$. We shall consider all the cases of possible configurations:

i) $p_2$ is not in $V$ and there is $\rho>0$ such that 
$$
B(p_2,\rho)\cap \d\Omega = \{ (x_1,x_2):\  x_2 = \psi^2 (x_1),\ |x_1+d| < \bar \rho\}.
$$
Moreover, $\frac d{dx_1} \psi^2(-d)$ exists and it equal to zero.
Then we repeat the construction we performed above, resulting in $\psi_n^2$, with the necessary changes. Finally, we end up with $\tilde\psi_n = \max\{\psi_n^1, \psi_n^2\}$. We may apply the construction in 1) to new $\tilde\Omega_n$, which is the sum of $\Omega$ and the epigraph of $\tilde\psi_n$.

ii) $p_2$ is  in $V$ and $p_2$  is the point of intersection of $\ell_k$ with $\ell_{k'}$, then we consider $\bar p_n$, the point of intersection of the line containing $\ell_{k'}$ with the graph of $\psi_n^1$. 
The
new auxiliary domain $\tilde\Omega_n$ is obtained by taking 
a region bounded by: the segment being a conv$(\ell_{k'}, \bar p_n)$, the graph of $\psi_n^1$ and the rest of $\d \Omega$ from $p_1$ till the endpoint of $\ell_{k'}$.

%the curve: the arc of $\d\Omega$ from $p_2$ the other endpoint of $\ell_{k'}$ till $(x_n,\psi(x_n))$ and the graph of $\psi_1$ and the segment $[p_2,\bar p]$.  Next, we may apply the construction in 1) to the new $\tilde\Omega_n$, which is the sum of $\Omega$ and the epigraph of .

iii)  $p_2$ is  in $V$ but the second condition in ii) does not hold. In other words,
there is $\rho>0$ such that 
$$
B(p_2,\rho)\cap \d\Omega = \{ (x_1,x_2):\  x_2 = \psi^2 (x_1),\ |x_1+d| < \bar \rho\}.
$$
Moreover, the left derivative of $\psi^2$ at $-d$ is negative. We can find a quadratic polynomial $Q(x_1) = \alpha x_1^2 + \beta x_1 + \gamma$ so that: $Q(-d) = 0$,  
$Q'(-d) = \frac d{dx_1}\psi^2(-d^-)$ and $\alpha >0$. Let us denote by $x_v$ the point, where $Q'$ vanishes.

We define
$$
\tilde\psi^2(x) = \left\{
\begin{array}{ll} 
 \psi^2(x) & x_1\in (-\bar\rho-d, -d),\\
 Q(x)\chi_{(-\bar\rho,x_v)}(x) + Q(x_v)\chi_{[x_v,0)}(x) & x_1\in [-d,0).
\end{array}\right.
$$
The  construction presented in 2) is applied to $p_1$ resulting with $\psi^1_n$. For $x\in( -\bar\rho-d, \bar\rho)$ we set 
$\tilde \psi^3_n (x)= \max\{ \tilde\psi_2(x), \psi^1_n(x)\}.$
This creates an intermediate domain 
$\tilde\Omega_n$, we apply step 1) to it.

%We may further assume that for sufficiently large $n$ the intersection $\d\Omega \cap B(x_1,\frac1n)$ is a graph of a convex function $\psi$. By Alexandrov's Theorem there is $\tilde x = (x_0,\tilde\psi(x_0))\in  B(x_1,\frac1n)$ such that $\tilde\psi''(x_0)$ exists and it is different from zero. For the sake of the argument we choose another 
%coordinate system such that  $\tilde x$ is its origin and the tangent to $\d\Omega \cap B(x_1,\frac1n)$ at $\tilde x$ the coordinate axis, so $\d\Omega \cap B(x_1,\frac1n) = \{(y,\psi(y)): \ y\in(a,b)\}$ and  $\tilde x=(0,0)$. Since $\tilde\psi''(x_0)\neq 0$ we infer that  $\psi''(0)=:\epsilon>0$. Thus, $\psi(y) >\frac \epsilon4 y^2$ for $y\in (0,b)$. 
%By assumption, there is a 
%positive $y_0$ such that for $y\in (y_0,b)$ the graph of $\psi$ is a line. Let us choose $\tilde  y :=(\frac 12(b+y_0), \frac\epsilon4(b+y_0)^2)$. 

%If the other endpoint of $\ell$ is in $V$ we perform the same construction resulting in $\tilde z$. We may require that the segment $[\tilde  y, \tilde  z]$ is parallel to $\ell$. Hence,  $[\tilde  y, \tilde  z]$ is not tangent to the graphs of the quadratic functions it meets at $\tilde  y$ and $\tilde  z$. 

%If the other endpoint of $\ell$, $v$ is not in $V$, then we consider segment $[\tilde  y,v]$ adjusting, if necessary the choice of $\tilde  y$, so that $[\tilde  y,v]$ is not tangent to the graph of the quadratic functions it meets at $\tilde  y$ nor to $\d\Omega$ at $v$.

%Once we constructed new segments and enlarged $\Omega$, we are in a position described in Step 1). 

4) We notice that the obtained domain $\Omega_n$ is strictly convex, since such modification are only performed on finitely many components.

\end{proof}
%\textcolor{red}{Mainly what we doing here is taking a flat part of $\partial \Omega$ and replace it with an arc without destroying convexity of the domain, this can be done because all the domain $\Omega$ is strictly contained in the hyperplane generated by the flat , so perturbing the flat part with a sufficient small perturbation shall not hurt the geometry of the domain. I think we over explained it above, we can be less explicit here, since the geometry is easy to conclude}

Once we constructed $\Omega_n$, we define $f_n:\d\Omega\to \bR$ as follows.
On each $c^n_k,$ for $(x,y) = (t,\psi^n_\ell(t))$, we set, 
\begin{equation}\label{defn}
f_n(x,y) \equiv f_n(t,\psi^n_\ell(t)) = f(t), \quad n\in\bN.
\end{equation}
In other words, $f_n$ has the same regularity properties as $f$ does. 

In case we constructed an intermediate region $\tilde\Omega_n$ we define $\tilde f_n:\d\tilde \Omega\to \bR$ using (\ref{defn}). Next, we introduce $f_n$ by  (\ref{defn}).

\begin{cor}\label{c0}
 If $f_n$ is defined above, then $\omega_f$, the continuity modulus of $f$, is also the continuity modulus of $f_n$. Moreover, if  $\pi:\bR^2\to \bar\Omega$ is the orthogonal projection onto a closed convex set $\bar\Omega$, then 
$ \pi_n := \pi|_{\d \Omega_n}$ is one-to-one
%$\pi^{-1} (\d\Omega) \cap \partial \tilde \Omega_l$ is uniquely defined 
and
$f_n \circ (\pi_n)^{-1}$ %(\d\Omega) \cap \partial \tilde \Omega_l) $ 
converges uniformly to $f$.
\end{cor}
\begin{proof}
 The argument is based on the observation that if $x_1, x_2\in \d\Omega_n$, then 
 $|\pi x_1 - \pi x_2 | \le  |x_1 -  x_2 |$. The details are left to the interested reader.
 
 The uniform convergence easily follows from the definition of $f_n$.
\end{proof}

We have to check that the distances, appearing in (\ref{A}), are well approximated through $f_n$, in the sense explained below. Let us assume  that the number of flat parts of $\d\Omega$ is finite, $\Omega_n$ are constructed in 
Lemma \ref{prap} and $f_n$ are defined above in (\ref{defn}). We assume that $I_i= [a_i, b_i]$ is a hump contained in  $\ell$. We denote the orthogonal projection onto the line containing $\ell$ by $\pi_\ell$. We set,
$$
\alpha^n_i := \pi_\ell^{-1}(a_i) \cap \d\Omega_n,\qquad
\beta^n_i := \pi_\ell^{-1}(b_i) \cap \d\Omega_n.
$$
We also denote by $y_i$, $z_i$ the points of $\d\Omega \setminus \ell$ defined by  (\ref{defyz}).

We have three possibilities for  $y_i$ and $z_i$: (i) both points belong to $\bigcup_{k\in\cI}\ell_k$, (ii)
one of these points belongs to $\bigcup_{k\in\cI}\ell_k$ while the
other one is in $\d\Omega\setminus\bigcup_{k\in\cI}\ell_k$  and the
last one is (iii) both points belong to
$\d\Omega\setminus\bigcup_{k\in\cI}\ell_k$. It is sufficient to
consider (i) and (iii), because they cover (ii) too. 

In case (i), since $\bar\Omega_n$ converges to $\bar\Omega$ in the Hausdorff distance, we conclude that 
$$
\lim_{n\to\infty} \dist(\alpha^n_i,y_i) + \dist(\beta^n_i,z_i)= \dist(a_i, y_i) + \dist(b_i, z_i) <|a_i -b_i|.
$$
As a result, the strict inequality holds for sufficiently large $n$.

In case (iii), we conclude that $y_i \in \ell'$ and $z_i \in \ell''$.  We consider the orthogonal projection $\pi_{\ell'}$, (resp. $\pi_{\ell''}$), onto  $\ell'$,  (resp. $\pi_{\ell''}$). We take
$$
\zeta^n_i := \pi_{\ell'}^{-1}(y_i) \cap \d\Omega_n,\qquad
\psi^n_i := \pi_{\ell''}^{-1}(z_i) \cap \d\Omega_n.
$$ 

Arguing as above, we conclude that 
\begin{equation}\label{rn-ap}
\lim_{n\to\infty} \dist(\alpha^n_i,\zeta^n_i) + \dist(\beta^n_i,\psi^n_i )
= \dist(a_i, y_i) + \dist(b_i, z_i) <|a_i -b_i|.
\end{equation}
%\textcolor{red}{But we are not destroying $\Omega$ away from the flat parts? so conditions ii and iii are avoided by our construction of $\Omega_n$}
As a result, the strict inequality holds for sufficiently large $n$.

Since the  number of flat parts is finite, our claim follows, i.e. we have shown:

\begin{cor}\label{c1} Let us suppose that $\Omega$ is open, bounded and convex and the number of flat parts of $\Omega$ is finite. We assume that
the function  $f\in C(\d\Omega)$ satisfies the admissibility 
condition \#2. Then,  (\ref{rn-ap}) holds for $\Omega_n$ and $f_n$ with sufficiently large $n$. \qed
\end{cor}

\begin{remark}\label{rem2}
One may wonder what is the role of intervals $(p_l,b')$ or $(a', p_r)$ contained in a flat part $\ell = [p_l, p_r]$ and such that $f$ restricted to $(p_l,b')$ or $(a', p_r)$ is constant. We claim that such intervals do not affect our argument. We know, how to proceed, if this is a hump.
%$f$ attains a local maximum (minimum) on $J=[p_l,b']\supset (p_l,b')$, then the construction of \cite{sternberg} applied to $\Omega_n$ shows that the length of $\d\{v_n \ge f([p_l,b'])\}$ converges to the distance between $p_l$ and $b'$, where $v_n$'s are taken from the proof of Theorem \ref{main1}. When $f$ is discontinuous at $p_l,$ $p_r,$ $b'$ or $a'$ we have to perform additional analysis. 

On the other hand, if such an interval  is not a hump, then %\marginpar{XXX}
the construction of \cite{sternberg} applied to $\Omega_n$ shows the optimal selection of the level sets. We can see that if $a'\in \d\{ u\ge f(a')\}$, (resp. $b'\in \d\{ u\ge f(b')\}$), then the other endpoint of $\d\{ u\ge f(a')\}$, (resp. $ \d\{ u\ge f(b')\}$), belongs to $\d\Omega \setminus \ell$.
%$\epsilon>0$ such that $f$ restricted to $\d\Omega\cap B(p_l,\epsilon) \cup  (p_l,b')$ (resp.  $\d\Omega\cap B(p_r,\epsilon) \cup  (a',p_r)$) is monotone, then the same kind of analysis as above applies.
\end{remark}

However, we may have an infinite number of flat parts of $\Omega$ or an infinite number of humps.
If so  (\ref{rn-ap}) contains an infinite number of conditions, which is not easy to  satisfy. This is why we introduce an intermediate stage of  construction of an approximation of $\Omega$. 

First, we consider the case of an infinite number of flat parts. We introduce a new piece of notation. 
We recall the assumptions we made on the flat parts, in case there is an infinite number of them.  They have a single accumulation point $p_0$ and almost all of them are on one side of $p_0$. If $\nu$ is the unit vector used to define the side of $p_0$, then for any $\rho>0$ we set
$$
B^+_\rho := \{x\in B(p_0,\rho):\ \nu\cdot(x-p_0) >0\},\qquad
B^-_\rho := \{x\in B(p_0,\rho):\  \nu\cdot(x-p_0) <0\}.
$$
We arrange $\ell_k$ so that $\dist(\ell_k, p_0)$ is a decreasing sequence.

\begin{lemma}\label{l-inter}
Let us suppose that $\Omega$ is convex with infinitely many flat parts $\{\ell_k\}_{k\in\cK}$, and almost all of them are on one side of a single accumulation point $p_0$. We also assume that $f\in C(\partial\Omega)$ satisfies admissibility conditions \#1 or \#2. In addition, if  $p_0$ is an endpoint of a flat part $\ell$, then we assume that 
$\ell$ has a finite  number of humps.

Then, there exists a sequence $\{\tilde \Omega^k\}_{l=1}^\infty$ of convex sets containing $\Omega$ and such that $\tilde \Omega^k$ has a finite number of  flat parts and $\overline{\tilde \Omega}^k$ converges to $\bar\Omega$ in the Hausdorff metric. Moreover, there are $\tilde f_k\in C(\partial \tilde \Omega^k)$ satisfying admissibility conditions \#1 or \#2 and
$$
\| \tilde f_k \|_{C(\partial \tilde \Omega^k)} \le \|  f \|_{C(\partial \Omega)}\quad\hbox{and}
\quad \omega_{\tilde f_k} \le {4} \omega_{ f},
$$
where $\omega_g$ denotes the continuity modulus of function $g$. In addition, if $\pi$ is the projection defined in Corollary \ref{c0}, then 
$\tilde \pi^k := \pi|_{\d \tilde \Omega^k}$ is one-to-one
%$\pi^{-1} (\d\Omega) \cap \partial \tilde \Omega^k$ is uniquely defined 
and
$\tilde f_k \circ (\tilde\pi^k)^{-1}$ %(\d\Omega) \cap \partial \tilde \Omega^k) $ 
converges uniformly to $f$.
\end{lemma}

\begin{proof}
%Since there are at most three points of $\partial\Omega$, where acute angles may form we can always consider $\rho$ such that no acute angles appear in $B(p_0,\rho)$ except possibly at $p_0$. \marginpar{\small purpose?}

We consider the sequence of all $\ell_k$  contained in $B(p_0, \rho)$ converging to $p_0$. We assume that they are so numbered that
$$
\dist (\ell_k, p_0) > \dist (\ell_{k+1}, p_0) .
$$

Under the above assumptions, we have the following situations:\\
(A) There is $\rho>0$ such that $B^+(p_0,\rho)\cap \partial\Omega $ is contained in a flat part $\ell$. If this happens, then we further restrict $\rho$ so that $B^+(p_0,\rho)\cap \ell$ contains no hump.\\
(B) For all  $\rho>0$ the set $B^+(p_0,\rho)\cap \partial\Omega $ is an arc, i.e. it does not contain any flat part.

Our goal is to construct approximations to $\Omega$. We first consider (A). 
%For this purpose we introduce another piece of notation. If $\tilde\ell$ is a line segment intersecting $\d\Omega$, but $\tilde\ell\cap\Omega=\emptyset$, then we set $H(\tilde\ell)$ to be a half-plane containing $\Omega$ whose boundary contains $\tilde\ell$. 

Since $\Omega$ is Lipschitz continuous, then there is a ball with radius possibly smaller then selected above and denoted again by $\rho$, a coordinate system  and a function $\psi:\bR \to \bR$ such that 
$$
B(p_0,\rho) = \{ (x_1, x_2): \ x_2 = \psi(x_1),\ x_1\in (-\gamma, \delta)\}.
$$
Keeping in mind the numbering of $\ell_k$, we introduced above, we proceed as follows. We consider only those $\ell_k$ that are contained in $B(p_0,\rho)$. Let us write $\ell_k = [p^k_l, p^k_r]$
where $p^k_l = \psi(x^k_l)$,   $p^k_r= \psi(x^k_r)$ and  $x^k_l < x^k_r$. If ${\bf t}$ is a line tangent to $\Omega $ and passing through  $p^k_r$, whose equation is $x_2 = \alpha^k x_1 + \beta^k$, then we set
$$
\psi^k (x_1) =\left\{
\begin{array}{ll}
 \psi(x_1), & x_1\in(-\gamma,  x^k_r],\\
 \alpha^k x_1 + \beta^k, & x_1\in (x^k_r,\delta).
\end{array}\right.
$$
Next, we define
\begin{equation}\label{Star}
 \Omega_\rho^k = \{(x_1, x_2): \ x_2 =\psi^k(x_1) \} \cap H(\ell)\cap B(p_0,\rho).
\end{equation}
We notice that $\bar \Omega_\rho^k \to \bar\Omega\cap \bar B(p_0,\rho)$ in the Hausdorff metric. We set 
$$
\tilde \Omega^k = \Omega_\rho^k \cup \Omega.
$$ 
Of course,  $\tilde \Omega^k $ is convex and it has a finite number of flat parts. We set  $q = \partial H(l_k) \cap \partial H(\ell)$, because we will need this point to define the boundary data. 

We  have to define $\tilde f_k$ on $\d\tilde \Omega^k$. We consider  a few subcases. Actually, it is enough to specify $\tilde f_k$ on  $\d\tilde \Omega^k \cap B(p_0,\rho)$.
We have the following situations:

(i) The intersection $B^+(p_0,\rho)\cap \partial\Omega $ contains $[p_0, b')$ such that $f$ on this interval is constant and $[p_0, b')$ in $\ell$ is maximal with this property. We can find $y'\in (\d\Omega\setminus \ell)\cap f^{-1}( f[p_0, b'])$ such that
$|y' - b'| = \dist (b', (\d\Omega\setminus \ell)\cap f^{-1}( f[p_0, b']))$. Then, we consider only such $k$  that $\dist(\ell_k, p_0) < | y'- p_0|$.

(ii) The flat part $\ell_{k}$ intersecting $B^-(p_0,\rho)$ has a hump $I_{k} = [a_{k}, b_{k}]$ which is such that interval $[b_{k} , z_{k} ]$ intersects $\ell$.  

(iii) The flat part $\ell_{k}$ intersecting $B^-(p_0,\rho)$ has a hump 
$I_{k} = [a_{k} , b_{k} ]$ which is such that interval $[b_{k} , z_{k} ]$ belongs to the arc connecting $b_{k}$ and   $z_{k}$ and contained in  $B^-(p_0,\rho)$.

(iv) The flat piece $\ell_k$ and $B^+(p_0,\rho)\cap \ell $ have no hump.

In all those cases we pick $\tilde \Omega^k$ given by (\ref{Star}). 
Here are the definitions of the boundary data $\tilde f_k$.

Case (ii) is the simplest. We set 
$$
\tilde f_k(x) = \left\{
\begin{array}{ll}
 f(x) & x \in [p_l^k, a_k] \cup (\d\tilde\Omega^k \cap B(p_0,\rho)) \setminus \d H(\ell_k) ,\\ 
 f(x) & x \in (\ell \cap B(p_0,\rho)) \setminus [p_0, z_k],\\
 f(a_k)& x \in [a_k, q]\cup [q, z_k].
\end{array} \right.
$$
Of course $\tilde f_k$ is continuous and it satisfies the admissibility conditions.

For cases (i) and (iv), we set
$$
\tilde f_k(x) = \left\{
\begin{array}{ll}
 f(x) & x \in \partial\tilde\Omega^k \setminus \d \Omega, \\%(\ell \cup \d H(\ell_k),\\
 f(p_0)& x \in [p_0,q],\\
\min\{ \omega(|q-x|) + f(p_0), f(p^k_r)\},& x \in [p^k_r, q]\quad \hbox{if }f(p_0) \le f(p^k_r),\\
\max\{f(p_0)-\omega(|q-x|), f(p^k_r)\},& x \in [p^k_r, q]\quad \hbox{if }
f(p_0) > f(p^k_r).
\end{array} \right.
$$
Since $|q-p^k_r| \ge |p_0-p^k_r|$ and $|f(p_0) - f(p^k_r)|\le \omega_f( |p_0- p^k_r|$, then the above definition is correct.

Finally, we construct $\tilde f_k$ if (iii) occurs. For this purpose we find $q'\in \d \tilde \Omega^k\cap  B(p_0,\rho)$ such that $| b_k- q'| = | b_k- z_k|$, if it exists. 
If there is no such $q'$, then we set $q':=p_0$. In both cases  we set,
$$
\tilde f_k(x) = \left\{
\begin{array}{ll}
 f(x) & x \in [\partial\tilde\Omega^k \setminus \d H(\ell_k)] \cup [p^k_l, b_k],\\
g_1(x) & x \in [b_k,q'],\\
g_2(x) & x \in [q', q],\\
f(x)& x \in \ell\cap B(p_0,\rho),\\
f(p_0)& x \in [q,p_0].
 \end{array}
 \right.
$$
Here, we use
$$
g_1(x)= \left\{
\begin{array}{lll}
\min\{ \omega(|x- b_k| + f(b_k), f(q')\},& x \in [b_k,q'], &\hbox{if } f(b_k) \le f(q'),\\
\max\{f(b_k) -\omega(|x- b_k| , f(q')\},& x \in [b_k,q'], &\hbox{if } f(b_k) > f(q'),
\end{array} \right.
$$
$$
g_2(x)= \left\{
\begin{array}{lll}
\min\{ \omega(|x- q' | + f(q'), f(p_0)\},& x \in [q', p_0], &\hbox{if } f(q') \le f(p_0),\\
\max\{f(p_0) -\omega(|x- q'| , f(p_0)\},& x \in [q', p_0], &\hbox{if }  f(q') >f(p_0).
\end{array} \right.
$$

If (B) occurs then we can face:\\
$l_{k}$ intersecting $\partial B^-(p_0, \rho)$ has a hump, then we proceed as in case (iii);\\
$l_{k}$ intersecting $\partial B^-(p_0, \rho)$ has no hump, then we proceed as in case (iv).

We have to estimate the modulus of continuity of $\tilde f_l$. We notice that if
$x,y\in \d \tilde \Omega^k$, then $|\tilde f_k(x) - \tilde f_k(y)| \le 4 \omega_f(|x-y|)$.

Finally, the construction is such that $\tilde f_l (\pi^{-1} (\d\Omega) \cap \partial \tilde \Omega_l) $ converges uniformly to $f$.
\end{proof}

\bigskip
Our construction of solutions will be performed in a few steps. We first treat a (curvilinear) polygon and $f$ having a finite number of humps. In this situation we can estimate the modulus of continuity of solutions to the approximate solutions on $\Omega_n$. This is done in Lemma below.

%\textcolor{red}{This should be a theorem}
\begin{lemma}\label{le2} Let us suppose that $f_n\in C(\d\Omega_n)$ is
  defined by (\ref{defn}). We  assume that  $f$ has
the  continuity modulus $\omega_f$ and it  has
  finitely many humps. Then, $v_n$ the unique solution to the
  least gradient on $\Omega_n$ with data $f_n$ exists and it is  continuous with
the  modulus of continuity $\tilde\omega_{f_n}$ and there exist $A, B>0$ independent of $n$, such that
$$
\omega_{v_n}(r) \le \omega_f\left( \frac r A + \frac {\sqrt r} B\right)=:\tilde\omega(r). 
$$
\end{lemma}

\begin{remark} We stress that $\tilde\omega$
  depends on $\omega_f$, $\Omega$ and the geometry of the data, but it
  does not depend on the number of humps.
\end{remark}

{\it Proof.\ } We notice that existence of $v_n$, solutions to (\ref{lg}) for each $\Omega_n$ and continuous $f_n$, follows from \cite[Theorems 3.6 and 3.7]{sternberg}.

In order to estimate $ \omega_{v_n}$, the modulus of continuity of $v_n$,
we consider a number of cases depending on the behavior
of flat pieces near the junction with the rest of $\d\Omega$. In
\cite{grs}, we could guess in advance the structure of the level set of
the solution. Here, it is much more difficult, so we use the fact that the  level set structure of $v_n$ is known. We set $E^n_t=\{ v_n(x)\ge t\}$. We know that $\d E^n_t$ is a sum of line segments. In general, we know that fat level sets may occur, so there may be points $x\in \Omega_n$, which do not belong to any $\d E^n_t$. 

So, 
the first situation we consider is:\\
{\it Case I}:\ \ $x_1, x_2\in \Omega_n$ belong to the
boundaries of the superlevel sets, i.e. there are $t_1,$ $t_2$ such
that $x_i\in \d E^n_{t_i}$, $i=1,2$. 

We have to estimate 
\begin{equation}\label{wew}
v_n(x_1) - v_n(x_2) = t_1 - t_2 = f_n(\bar x^{t_1}) - f_n(\bar x^{t_2}), 
\end{equation}
for properly chosen points $\bar x^{t_i} \in \d\Omega_n\cap \d E^n_{t_i}$, $i=1,2$, 
in terms of the continuity modulus of $f$. Existence of $\bar x^{t_i}$ is guaranteed by \cite{sternberg}.

We have to estimate the distance between the intersection of $\d\Omega_n$ and $E^n_{t_i}$, $i=1,2$.
We will consider a number of subcases. Here is the first one:

(a) There
exist flat parts of $\Omega$, $\ell_1$, $\ell_2$, which are parallel and  such that $E^n_{t_i}$, $i=1,2$, intersect both of them. %We notice that an easier than stated above is to estimate the distance between the intersection points of $\ell_j$, $j=1,2$ with $E^k_{t_i}$, $i=1,2$. In the computations below  $\ell_j$, $j=1,2$ will be two line segments, not necessary the flat parts. In the next step we show how this piece of informatrion will be used to find the desired estimate.
We will use the following shorthands, $\d E^n_{t_i} =: e_i$, $i=1,2$, see also Fig. 1.

\begin{figure}
{\centering
\includegraphics[width=10cm,height=6cm]{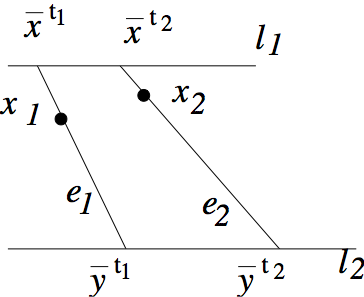}  
  \caption{Case a.}
  }
  \end{figure}

The first observation is obvious,
$$
|x_1 -x_2| \ge\min\{ \dist(x_1, e_2), \dist(x_2, e_1)\} \ge  \dist(e_1, e_2).
$$
Let us write $\{x^{t_1},x^{t_2}\} = \ell_1\cap( e_1 \cup e_2)$ and 
$\{y^{t_1},y^{t_2}\} = \ell_2\cap( e_1 \cup e_2)$.
If $\alpha_i$ is the angle formed by $e_i$ and $\ell_1$ or $\ell_2$, $i=1,2$, then 
$$
\dist(e_1, e_2) \ge \min \left\{ 
|y^{t_2} - y^{t_1}|\sin \alpha_2,|x^{t_2} - x^{t_1}|\sin \alpha_1
\right\}.
$$
We may estimate $\alpha_1$, $\alpha_2$ from below by $\alpha$, such that 
$$
\tan \alpha = \frac{\dist(\ell_1, \ell_2)}{\diam(\pi_2 \ell_1 \cup \ell_2)} \ge
\frac{\dist(\ell_1, \ell_2)}{\diam(\Omega)},
$$
where $\pi_2$ is the orthogonal projection onto the line containing $\ell_2$.

We may continue estimating the right-hand-side (RHS) of
(\ref{wew}). If $|x^{t_1} - x^{t_2}| <|y^{t_1} - y^{t_2}|$, then we
choose for $\bar x^{t_i}$ the point in $\d\Omega_n\cap \d E_{t_i}$, which is closer to $\ell_1$. Then, 
\begin{equation}\label{p-wacek}
|v_n(x_1) - v_n(x_2)| = |f_n(\bar x^{t_1}) - f_n(\bar x^{t_2})| = 
|f(\pi_1 \bar x^{t_1}) - f(\pi_1 \bar x^{t_2})|
\le 
\omega_f(| \pi_1\bar  x^{t_1} - \pi_1\bar  x^{t_2}|),
\end{equation}
where $\pi_1$ is the orthogonal projection onto the line containing $\ell_1$. We also use here the definition of $f_n$. We notice that our construction yields,
\begin{equation}\label{wacek}
|\pi_1 \bar x^{t_1 } - \pi_1\bar  x^{t_2}|\le |x^{t_1} - x^{t_2}|\le  |x_1 - x_2|/\sin \alpha. 
\end{equation}
Hence,
$$
 |v_n(x_1) - v_n(x_2)| \le \omega_f(| x_1 - x_2|/\sin\alpha).
$$
%We will see that (\ref{wacek}) is what we really need.

If  $|x^{t_1} - x^{t_2}| \ge |y^{t_1} - y^{t_2}|$ we continue in a similar fashion. Namely, we 
choose the points in $\d\Omega_n\cap \d E_{t_i}$, which are closer to $\ell_2$
and we call them $\bar  y^{t_i}\in e_i$, $i=1,2$. We conclude that
\begin{equation}\label{wacek2}
|\pi_1 \bar y^{t_1 } - \pi_1\bar  y^{t_2}|\le |y^{t_1} - y^{t_2}|.
\end{equation}
Hence,
$$
 |v_n(x_1) - v_n(x_2)| \le \omega_f(| y^{t_1} - y^{t_2}|/\sin\alpha).
$$
As a result, we reach
\begin{equation}\label{om_1}
 |v_n(x_1) - v_n(x_2)| \le \omega_f(|x_1 -x_2|/\sin \alpha).
\end{equation}

(b) The next subcase is, when $e_1$ and $e_2$ intersect $\ell_1$ and $\ell_2$, which are not parallel and $\ell_1 \cap \ell_2 =\emptyset$, see Fig. 2. We proceed as in subcase (a).

\begin{figure}
{\centering
\includegraphics[width=10cm,height=8cm]{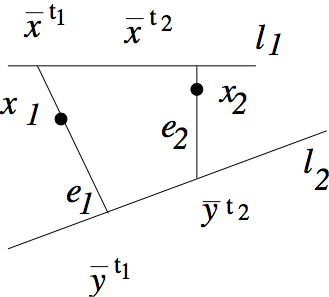}  
  \caption{Case b.}
  }
  \end{figure}
%\begin{figure}[ht!]  \includegraphics[width=122mm]{fig1_fig2.jpg}  \end{figure}

We have to estimate $|x_1-x_2|$ from below. Of course we have,
$$
|x_1 -x_2| \ge \min\{\dist(x_1, e_2), \dist(x_2, e_1)\}\ge
\min\{\dist(x^{t_1}, L(e_2)), \dist(y^{t_1}, L(e_2))\},
$$
where $L(v)$ is the line containing a (nontrivial) line segment
$v$. We notice that if $\beta_{ij}$ is the angle, which $e_i$ forms with $\ell_j$, then
$$
\sin \beta_{11}=\frac{\dist(x^{t_2},e_1)}{|x^{t_2}-x^{t_1}|}, \quad
\sin \beta_{12}=\frac{\dist(y^{t_2},e_1)}{|y^{t_2}-y^{t_1}|}, \quad
\sin \beta_{21}=\frac{\dist(x^{t_1},e_2)}{|x^{t_2}-x^{t_1}|}, \quad
\sin \beta_{22}=\frac{\dist(y^{t_1},e_2)}{|y^{t_2}-y^{t_1}|}.
$$
We want to find an estimate from below on $\beta_{ij}$. We can see that $\beta_{ij} \ge \beta_m$, $i,j=1,2$, where
$$
\sin\beta_m =\min\left\{ 
\frac{\dist(\d \ell_1, \ell_2)}{|\pi_2\ell_1|}, 
\frac{\dist(\d \ell_2, \ell_1)}{|\pi_1\ell_2|}\right\},
$$
where $\pi_i$ is the orthogonal projection onto the line $L(\ell_i)$, $i=1,2$.

Combining these estimates, we can see that
$$
|x_1-x_2| \ge \min \{ |x^{t_2}-x^{t_1}|, |y^{t_2}-y^{t_1}|\} \sin \beta_m.
$$
In this way we obtain
$$
v_n(x_1) - v_n(x_2) = t_1 - t_2 
$$
where $t_i = f(\bar x^{t_i})$ or $t_i = f(\bar y^{t_i})$, $i=1,2$ and $\bar x^{t_i}$, $\bar y^{t_i}$, 
$i=1,2$ are defined as in step (a). 
Arguing as in step (a), we reach the same conclusion as in (\ref{wacek}) or (\ref{wacek2}).
Hence, 
\begin{equation}\label{om_2}
|v_n(x_1) - v_n(x_2)| \le \omega_f(\min\{|x^{t_2}-x^{t_1}|, |y^{t_2}-y^{t_1}|\}) \le 
\omega_f( |x_2-x_1|/\sin\beta_m). 
\end{equation}
(c) The next subcase is when $e_1$ and $e_2$ intersect $\ell_1$ and $\ell_2$, which are not parallel and $\ell_1 \cap \ell_2 =\{V\}$, see Fig. 3. 
\hspace{4cm}
%{\centering
\begin{figure}[h]
{\centering
\includegraphics[width=10cm,height=6cm]{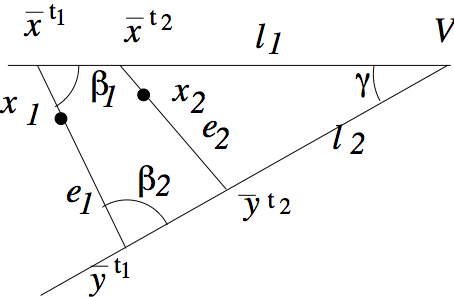}  
  \caption{Case c.}
  }
  \end{figure}
%}

%\begin{figure}[ht!]  \includegraphics[width=122mm]{fig3.jpg}  \end{figure}

If this happens, then the admissibility conditions restrict positions of $y^{t_1}$, $y^{t_2}\in \ell_2$, relative to $x^{t_1}$, $x^{t_2}$. Indeed, we can find $i_o\in \cI$ so that
$$
\min\{\dist(a_{i_o},V),\dist(b_{i_o},V)\} = \dist(I_{i_o},V)
=\min \{\dist(I_j,V):\ I_j\subset \ell_1\}.
$$
Due to condition (\ref{A}), there are $z_{i_o}, w_{i_o}\in \d\Omega\setminus\ell$ such that the distances in  (\ref{A}) are attained there, i.e.
$$
\dist(a_{i_o},z_{i_o})+\dist(b_{i_o},w_{i_o}) =
\dist(a_i, f^{-1}(e_i)\cap(\d\Omega\setminus I_i)) + \dist(b_i, f^{-1}(e_i)\cap(\d\Omega\setminus I_i)) < |a_i-b_i|.
$$
We may also assume that
$$
\dist(b_{i_o},V) < \dist(a_{i_o},V)\quad\hbox{and} \quad
\dist(w_{i_o},V) < \dist(z_{i_o},V).
$$
We consider a triangle $T:= \bigtriangleup(V,b_{i_o},w_{i_o})$ and the following cases (i) none of $x_1$, $x_2$ belong to $T$, (ii) just one of $x_1$, $x_2$ belongs to $T$, (iii)  $x_1$ and $x_2$ belong to $T$.

It is obvious that (i) reduces to (b). Situation in (ii) may reduced to (b) and (iii) below by introducing an additional point $x_3$, the intersection of $[x_1, x_2]$ with 
$[b_{i_o},w_{i_o}]$.  Finally, we have to pay attention to (iii) when
positions of $y^{t_1}$, $y^{t_2}\in \ell_2$ are not restricted. In  this case, we proceed as in \cite{grs}. We notice that
$$
|x_1-x_2| \ge \dist(x_2, e_1) = \min\{ \dist(x^{t_2}, e_1), \dist(y^{t_2}, e_1)\}.
$$
In addition, if $\beta_i$ is the angle formed by $e_1$ with $\ell_i$, $i=1,2$, then we notice
$$
\sin\beta_1 = \frac{\dist(x^{t_2}, e_1)}{|x^{t_2}-x^{t_1}|},\qquad
\sin\beta_2 = \frac{\dist(y^{t_2}, e_1)}{|y^{t_2}-y^{t_1}|}.
$$
While estimating $\sin\beta_i$, $i=1,2$, we have to take into account that $\ell_1$ and $\ell_2$ form 
an angle $\gamma$. Thus,
$$
\sin\gamma = \frac{d^y}{|y^{t_2}-y^{t_1}|},
$$
if $|y^{t_2}-y^{t_1}|> |x^{t_2}-x^{t_1}|$ and
$$
\sin\gamma = \frac{d^x}{|x^{t_2}-x^{t_1}|},
$$
in the opposite case. In these formulas, $d^y$ (resp. $d^x$) denotes the length of the orthogonal projection of the line segment $[y^{t_2},y^{t_1}]$  (resp. $[x^{t_2},x^{t_1}]$) on the line perpendicular to $\ell_1$  (resp. $\ell_2$). 
Thus, we can estimate $\sin \beta_i$, $i=1,2$, below as follows,
$$
\sin \beta_1 \ge \frac{d^y}{\diam(\Omega)} = \frac{\sin\gamma |y^{t_2}-y^{t_1}|}{\diam(\Omega)},
\qquad
\sin \beta_2 \ge \frac{d^x}{\diam(\Omega)} = \frac{\sin\gamma |x^{t_2}-x^{t_1}|}{\diam(\Omega)}.
$$
As a result, 
$$
|x_1-x_2| \ge \frac{\sin\gamma}{\diam\Omega}|x^{t_2}-x^{t_1}| |y^{t_2}-y^{t_1}|.
$$
Hence,
$$
\sqrt{\frac{\diam\Omega}{\sin\gamma}} \sqrt{|x_1-x_2|} \ge 
\min \{|x^{t_2}-x^{t_1}|, |y^{t_2}-y^{t_1}|\}.
$$
Arguing as in parts (a) and (b) we come to the conclusion that
\begin{equation}\label{om_3}
 |v_n(x_1) - v_n(x_2)| \le \omega_f(A \sqrt{|x_1 -x_2|}), %/\sin \alpha).
\end{equation}
where $A = \sqrt{\diam\Omega}/ \sqrt{\sin\gamma}$.

Subcase (d): 
$e_1$ and $e_2$, defined earlier, intersect $\ell_1$. In addition, there are two different flat parts $\ell_2$ and $\ell_3$ intersecting $e_1,$ $e_2$ i.e., $e_1 \cap\ell_2 \neq \emptyset$ 
and $e_2\cap \ell_3 \neq \emptyset$. We will reduce this situation to:\\
 (b) when $\ell_1 \cap (\ell_2 \cup \ell_3) = \emptyset$
or (c) $\ell_1 \cap (\ell_2 \cup \ell_3) = \{V\}$ or (d1) the intersection 
$\ell_1 \cap (\ell_2 \cup \ell_3)$ consists of two points.

The reduction is as follows. Let us suppose that $\ell_2 \cap \ell_3= \{P\}$.
We take $t_3$ such that $\d E^n_{t_3} \cap \d\Omega$ contains $P$. If there
is no such $t_3$, then we are in the situation of Case II considered below. 
We call by $e_3$ a component of $\d E^n_{t_3} $ containing $P$. Now, $e_3$
intersects segment $[x_1, x_2]$ at $x_3$ and
$\ell_1$ at $x^{t_3}$. Now, pairs $x_1$, $x_3$ and $x_3$, $x_2$ fall into the 
known category (b) or (c) or we have to proceed iteratively to reach them.

The iterative procedure, indicated above, is necessary when $\ell_2$ and $\ell_3$ are disjoint.  

Case (d1) can be reduced to the previous ones. Namely, If $\ell_2\cap \ell_3 =\{P\}$, then we consider the level set containing $P$. If there is $\tau$ such that $P\in \d E^n_\tau$, then we can  take $x_3\in \d E^n_\tau\cap [x_1,x_2]$ and in order to estimate $|v_n(x_1) - v_n(x_2)|$, we will use the triangle inequality
$$
|v_n(x_1) - v_n(x_2)| \le |v_n(x_1) - v_n(x_3)| + |v_n(x_3)- v_n(x_2)|
$$
and the observation that the points $x_1$, $y$ and $x_2$, $y$ belong to the categories we have already investigated.

Subcase (e): $e_1$ and $e_2$, defined earlier, intersect $\ell_1$ a flat part and 
$e_1$ and $e_2$ intersect $\cC$ a connected component of 
$\d\Omega \setminus \bigcup_{i=1}^K \ell_i$, and $\cC \cap \ell = \emptyset$.
In such a situation we proceed as in case (b), but instead of $\ell_2$, we 
consider all cords of arc $\cC$. We can estimate 
$|v_n(x_1) - v_n(x_2)|$ as in (\ref{om_1}).

We may assume that $\bar x^{t_1}, \bar y^{t_1} \in\ell $ and  $\bar x^{t_2}, \bar y^{t_2} 
\in \d\Omega \setminus \bigcup_{i=1}^K \ell_i$. We define $\tilde\ell_1 =[\bar x^{t_1},\bar y^{t_1}]$ and
$\tilde\ell_2 =[\bar x^{t_2},\bar y^{t_2}]$. Once we introduced $\tilde\ell_1$ and  $\tilde\ell_2,$ they will play the role of $\ell_1$ and  $\ell_2$. We recognize one of the subcases (a) to (e). We notice that the situation simplifies a bit since $\bar x^{t_2}, \bar y^{t_2}\in \Omega$.

Subcase (f) occurs when both $e_1$ and $e_2$  intersect $\d\Omega \setminus \bigcup_{i=1}^K \ell_i$. We proceed as in subcase (e) and we notice that $\bar x^{t_i}, \bar y^{t_i}\in \Omega$, $i=1,2$.

{\it Case II} occurs when $x_1$ belongs to $\d E^n_t$, while for no real $s$, point $x_2$ belongs to $\d E^n_s$. Since $v_n$ is continuous, thus $v_n(x_2)=\tau$ is well-defined. We take $x_3 \in \d E^n_\tau\cap [x_1,x_2]$. As a result, couples $x_1, x_3$ and $x_3, x_2$ fall into one of the investigated categories above. 

The final {\it Case III} is when neither $x_1$ nor $x_2$ belong to any $\d E^n_t$. Let us assume that $t_1>t_2$ (in case $t_1 =t_2$ there is nothing to prove). 
We take $x_3 \in [x_1,x_2]\cap \d
E^n_{t_1}$. Clearly, the present case reduces to the previous one, because $v_n(x_1) = v_n(x_3)$ and the couple $x_2$, $x_3$ belongs to the Case II.\qed
%\textcolor{red}{In above proof, the cases should be clearly highlighted especially the sub cases.}

\begin{theorem}\label{main1} Let us suppose that $\Omega$ is convex
  and $f\in %BV(\d\Omega)\cap 
  C(\d\Omega)$. In addition, $\partial\Omega$ may have
  countably many flat parts $\{\ell_k\}_{k\in \cK}$. If this happens, then they are on one
  side of their single accumulation point $p_0$.
 If $f$ satisfies  the admissibility conditions \#1 or \#2 on each
 flat part $\ell$ of $\d\Omega$, then there is a continuous solution to the least gradient problem.
\end{theorem}

\begin{remark} We assume that the flat parts are on one side of $p_0$
  just for the sake of convenience. With the same tools, we can handle
  also a finite number of accumulation points.
\end{remark}
\begin{proof} {\it Step 1.} We assume initially that $\Omega$ has a
  finite number of flat parts and we have a  finite number of humps.
We use Lemma \ref{prap} to find a sequence 
of  strictly convex regions,  $\Omega_n$, approximating  $\Omega$. The continuity modulus of 
the boundary function $f$ is denoted by $\omega_f$. 
We notice that all $f_n$ have continuity modulus $\omega_f$. By
Corollary \ref{c1} functions $f_n$ satisfy the admissibility conditions.

We notice that by classical result, \cite{sternberg}, there exists a
unique solution, $v_n$ to the least gradient problem (\ref{lg}) on
$\Omega_n$ with data $f_n$. 

By the maximum principle, see \cite{sternberg}, sequence $v_n$ is
uniformly bounded and one can show %a uniform estimate on 
$\int_{\Omega_n} |Dv_n|\le M<\infty.$ 
Now, we set,
$$
u_n = \chi_\Omega v_n.
$$
From \cite[Proposition 4.1]{grs} we know that $u_n$ are least gradient functions.

Since functions $u_n$ are uniformly bounded and due to Lemma \ref{le2} they have the common
continuity modulus $\tilde\omega$, there is a subsequence (not
relabeled) uniformly converging to $u$. The uniform convergence
implies convergence of traces, i.e. $Tu_n$ goes to $Tu$. Since $Tu_n$
tends to $f$, we shall see that $Tu = f$. Indeed, if $x\in \d\Omega$, then
\begin{eqnarray*}
 |u_n(x) - f(x)| & \le & | v_n(x) - v_n \circ (\pi_n)^{-1}(x) | + 
 | v_n \circ (\pi_n)^{-1}(x)- f(x)| \\
 &=& | v_n(x) - v_n \circ (\pi_n)^{-1}(x) | +
 |f_n((\pi_n)^{-1}(x))- f(x)|.
\end{eqnarray*}
Since $(\pi_n)^{-1}(x)$ goes to $x$ as $n\to\infty$ and we can use the last part of Corollary \ref{c0}, we conclude that the right-hand-side above converges to zero, so $Tu = f$.

Moreover, the uniform convergence of $u_n$ implies the convergence of
this sequence to $u$ in $L^1$. Hence, by classical results,
\cite{miranda}, we deduce that $u$ is a least gradient function. Since
it satisfies the boundary data, we deduce that $u$ is a solution to the
least gradient problem. Moreover, the modulus of continuity of $u$ is $\tilde\omega$.

{\it Step 2.} 
Now, we relax the assumption on $f\in C(\d\Omega)$ and we admit it has an infinite number of humps. We denote its continuity
modulus by $\omega_f$.

We assume that only one flat side $\ell$  has infinitely many humps. We do this for the sake of simplicity of the argument. We will find functions  $g_n\in C(\d\Omega)$, which converge uniformly to $f$ and each of them satisfies the admissibility conditions $\#1, 2$ and  has finitely many humps.

%In order to proceed we have to discover the structure of the accumulation points of the 
%If $f$ has an 
%infinite number of humps $\{I_i\}_{i\cI}$ contained in the  flat part $\ell$.
We claim that due to the admissibility  conditions $\#1$ and $\#2$ the humps $\{I_i\}_{i\cI}$ may accumulate only at the endpoints of $\ell$. Let us suppose the contrary and $a$ is an accumulation point of $I_l = [a_l, b_l]$, $l \to \infty$, i.e. $|b_l - a_l|\to 0$ and $a_l \to a$. Since $a$ is not any endpoint of $\ell$, then $\dist(a,\d\Omega\setminus\ell)>0$, but this violates the admissibility condition \# 2.

Let $a$ be an endpoint of $\ell$ which is an accumulation point of $\{I_i\}_{i\cI}$ (if the humps accumulate also at the other endpoint, then we proceed similarly).

%Then, since $f$ satisfies admissibility condition \#2, then these humps have lengths that converge to zero and must accumulate at the edge of $\ell$. Let $a_1,\cdots , a_N$ be the edges where humps accumulate to, and corresponding to sides $\ell_1,\cdots, \ell_N$.

Let $b^n$ be  the point on $\ell$ such that $|a-b^n|<1/n$.
We define the sequence $g_n: \d\Omega\to \bR$ as follows, %for $x\in\ell$ 
we set
$$
g_n(x)=\left\{
\begin{array}{ll}
 f(x) & x \in \d\Omega\setminus[a,b^n],\\
 \max\{ f(a), f(b^n)-\omega_f(|x - b^n|)\}& x \in [a,b^n]\quad\hbox{and }f(b^n) > f(a),\\
 \min\{ f(a), f(b^n)+\omega_f(|x - b^n|)\}& x \in [a,b^n]\quad\hbox{and }f(b^n) \le f(a)
\end{array}
\right.
$$
Of course,  $g_n$ are continuous and they converge uniformly to $f$ on $\p \Omega$. We may easily estimate the continuity modulus of $g_n$  by $2\omega_f$.
%Indeed,
%We shall prove that $g_n$ has modulus of continuity equal to $w_f$ on $[a_ib_i^n]$.
%In fact
%If $x$ is not on $[a_i,b_i^n]$ then $g_n-f=0$ otherwise
%$$g_n(x)-f(x)=\omega_f(x)+f(x)-f(b_i^n)-\omega_f(b_n)\leq 2\omega_f(|x-b_n|)\leq 2\omega_f(1/n)\to 0
%$$Hence $g_n\to f$ uniformly to $f$ on $\p \Omega$.

Let $u_n$ be the solution to the least gradient problem corresponding to $f_n$, then due to Lemma \ref{le2} $u_n$ is continuous up the boundary uniformly bounded and with modulus of continuity equal to $2\tilde\omega$. Hence, by Arzela-Ascoli Theorem, $u_n$ converges uniformly to a function $u$ which, by \cite{miranda}, is a least gradient function. 
The uniform convergence
implies convergence of traces, i.e. $Tu_n$ goes to $Tu$. Moreover, we
can check that $Tu = f$ exactly as in Step 1. In the proof carried out
there we use  Corollary \ref{c0}, now   in its place, we refer to 
%Indeed, if $x\in \d\Omega$, then
%\begin{eqnarray*}
% |u_n(x) - f(x)| & \le & | v_n(x) - v_n \circ (\tilde\pi_n)^{-1}(x) | + 
% | v_n \circ (\tilde\pi_n)^{-1}(x)- f(x)| \\
% &=& | v_n(x) - v_n \circ (\tilde\pi_n)^{-1}(x) | +
% |f_n(\tilde\pi_n)^{-1}(x))- f(x)|.
%\end{eqnarray*}
%Since $\tilde\pi_n)^{-1}(x)$ goes to $x$ as $n\to\infty$ and we can use 
the last part of Lemma \ref{l-inter}.
%, we conclude that the right-hand-side above converges to zero, so $Tu = f$. 

%Now 
%$$Tu(x)=\lim_{y\to x, y\in \Omega} u(y)=\lim_{y\to x, y\in \Omega}\lim_{n\to \infty} u_n(y)=\lim_{n\to \infty} \lim_{y\to x, y\in \Omega}u_n(y)=\lim_{n\to \infty}g_n(x)=f(x).$$

{\it Step 3.} We assume  we have an  infinite number of flat
parts. If this happens, we use an intermediate step of approximation. We do not
make any assumption on the  number of humps. Due to Lemma
\ref{l-inter} there is a sequence $\tilde \Omega_k$ of convex bounded
sets with a finite number of flat parts and such that $\overline{\tilde
\Omega}_k$ converges to $\bar \Omega$ in the Hausdorff
metric. Moreover, we have  $\{\tilde f_k\}_{k=1}^\infty$ which is a sequence of data satisfying the
admissibility conditions on $\partial \tilde \Omega_k$.The
moduli of continuity of $\tilde f_k$ are commonly  bounded by
$4 \omega_f$.

Due to Step 2, there exists a sequence $\{\tilde u_k\}_{k=1}^\infty$ of
solutions to the least gradient problems in  $\tilde \Omega_k$ with
the common bound on the modulus of continuity, due to Lemma
\ref{le2}. Thus, the sequence $\{\tilde u_k\chi_\Omega\}_{k=1}^\infty$
is bounded in $C(\Omega)$ with a common modulus of continuity. Hence,
we can extract a convergent subsequence. We will call the limit by
$u$. Arguing as in Step 1. we deduce that $u$ is a least gradient
function and it has the correct trace.
\end{proof}

Once we proved existence, we address the problem of uniqueness of solutions. In \cite{gorny2}, the author studied the problem of uniqueness of solutions to the least gradient problem understood in the trace sense, i.e. as here. The cases of non-uniqueness are classified there and related to the possibility of different partition of `fat level sets', i.e. level sets with positive Lebesgue measure, and with the possibility of assigning different values there. In case of continuous data and solutions, we do not have any freedom to choose values of solutions on fat level sets. Thus, \cite[Theorem]{gorny2} implies the following statement. 

\begin{cor}
 Solutions constructed in Theorem \ref{main1} are unique.
\end{cor}

Now, we are ready to deal with discontinuous data.

%\begin{theorem}
%Let us suppose that $\Omega$ is convex and $f\in BV(\d\Omega)$. We
%assume that $f$ satisfies the admissibility conditions. Then,
%there exists a solution to the least gradient problem.
%\end{theorem}

\section{The case of discontinuous data}\label{sDdata}

In this section, we relax the continuity condition  and study \eqref{lg} when $f \in BV$. In this case $f$ might have at most  countably many jump discontinuity points. 
The technique we use here permits us to consider $f$ with countably many discontinuity points. %in $\d\Omega\setminus(\bigcup_{k\in \cK} \ell_k)$. However, in order to avoid unnecessary difficulties we restrict our attention to $f$ which are continuous on $\d\Omega\setminus(\bigcup_{k\in \cK} \ell_k)$.
We assume that $f$ satisfies the admissibility condition given in Definition \ref{admDC} and \ref{admDC2}.  We stress that they slightly different from  \ref{admC1} and \ref{admC2}.

We start by showing the following results that will be needed later in the construction of the solution. In fact, this a result borrowed from \cite{NR}.

\begin{lemma}\label{Aprox1}
Let $f$ be a monotone function in  $\mathbb R$, then there exist two sequences of continuous functions $g_n$ and $h_n$, such that:
\begin{enumerate}
\item $g_n$ and $h_n$ are monotone with the same monotonicity as $f$ and $g_n(x)\le f(x) \le h_n(x)$;
\item $\left\{g_n(x)\right\}$ is an increasing sequence, and $\left\{h_n(x)\right\}$ is a decreasing sequence;
\item $g_n$ and $h_n$ converge to $f$ at continuity points of $f.$
\end{enumerate}
\end{lemma}

\begin{proof}
It is enough to show 1.--3. for  $f$ increasing.
Let $\varphi_{1/n}$ be the standard approximation to the identity with support in $[-\frac 1n, \frac 1n]$. We consider the mollified sequence 
$$g_n(x)= f\ast \varphi_{1/n}(x-\frac{1}{n})=\int_{\mathbb R} f(x-\frac{1}{n}-\xi)\varphi_{1/n}(\xi)\,d\xi=\int_{-1/n}^{1/n}f(x-\frac{1}{n}-\xi)\varphi_{1/n}(\xi)\,d\xi.$$
Functions $g_n$ are continuous and the sequence converges pointwise to $f$ at continuity points.
Since $\varphi_1 \ge 0$ and $f$ is increasing then $g_n$ is an increasing function for every $n\in \mathbb N$.
It remains to show that for every $x$ the sequence $\{g_n(x)\}$ is increasing. In fact using the change of variable $z=ny$, we get 
\begin{eqnarray*}
 g_n(x)&=&\int_{\mathbb R} f\left(x-\frac{1}{n}z\right)\varphi_1(z)dz=\int_{-1}^{1}f\left(x-\frac{1}{n}(1+z)\right)\varphi_1(z)\,dz\\
 &\leq& \int_{-1}^1 f\left(x-\frac{1}{n+1}(1+z)\right)\varphi_1(z)\,dz=g_{n+1}(x).
\end{eqnarray*}
Letting $h_n(x)= f\ast \varphi_{1/n}(x+\frac{1}{n})$, we can prove similarly 1.--3.
\end{proof}

\begin{cor}\label{Aprox2}
Let $f\in BV$, then there exist two sequences of continuous functions $g_n$ and $h_n$, such that
\begin{enumerate}
\item $\left\{g_n(x)\right\}$ is an increasing sequence, and $\left\{h_n(x)\right\}$ is a decreasing sequence.
\item $g_n$ and $h_n$ converge to $f$ at continuity points of $f$.
\end{enumerate}
\end{cor}

\begin{proof}
Since $f$ is in $BV$, we can write 
$f=f^+ - f^-$, with $f^+, f^-$ increasing functions. Using Lemma \ref{Aprox1}, we conclude the proof of the corollary.
\end{proof}

\subsection{Data with a finite number of humps}
% \textcolor{blue}{THE PROOF of The existence Lemmas and the Approximation works for infinitely many discontinuities, the title should be for finitely many humps.}
%\textcolor{red}{We should emphasize that we are using different admissibility conditions that are mentioned in section 2, we should refer to the definitions}
We treat separately the cases of finite and infinite number of jumps. This necessity is apparent to distinguish the cases, when we approximate the data. Our point is
to  make sure that the approximation process in Lemma \ref{Aprox1} preserves the main feature of the data. We have:
\begin{lemma}\label{gnadm2}
 Let us suppose that $f\in BV(\d\Omega)$ satisfies admissibility condition \#2. Then, the approximating sequences $g_n$ and $h_n$ satisfy  admissibility condition \#2 for continuous functions.
\end{lemma}
\begin{proof}
We defined humps so that $\bar I_i =[a_i, b_i]$ and $f(a_i, b_i) = e_i$.
By Lemma \ref{Aprox1} $g_n(a_i)<f(a_i)$, hence 
$[\alpha^n_i, \beta^n_i] = g^{-1}(e_i) \cap I_i \subsetneq f^{-1}(e_i) \cap I_i$ but continuity of $f$ and $g_n$ on $I_i$ implies that 
$$
\alpha^n_i \to a_i,\qquad \beta^n_i \to b_i.
$$
Thus, (\ref{DA}) holds for sufficiently large $n$.
\end{proof}

The admissibility condition \#1 is a bit more difficult to handle. However, we prove the following observation.
\begin{lemma}\label{gnadm1}
 Let us suppose that $f\in BV(\d\Omega)$ satisfies admissibility condition \#1. Then, the approximating sequences $g_n$ and $h_n$ can be modified to satisfy  admissibility condition \#1 for continuous functions.
\end{lemma}
\begin{proof}
 We will use the following observation. If $f$ is monotone on $(\alpha - \epsilon, \beta +\epsilon)$ and $\varphi_\epsilon$ is the standard mollifying kernel with support in $[-\epsilon, \epsilon]$, then $f *\varphi_\epsilon$ is monotone on $[\alpha, \beta]$. Thus, if condition (ii) in Definition \ref{admDC} holds, then for sufficiently large $n$ functions $g_n$ and $h_n$ are monotone restricted to $I_i \cup \bigcup_{x_0\in\d I_i}B(x_0,\epsilon)\cap\d\Omega$.
 
Let us now suppose that condition (i) in Definition \ref{admDC} holds. In this case, we have to proceed differently. 
Due to the observation made at the beginning of the proof function $g_n$ is monotone on $[a_i+\frac 2n, b_i -\frac 2n]$. We take $g^{-1}(g(a_i))\cap [a_i,a_i+\frac 2n]$ and its element which is the farthest from $a_i$ is called  $d$. Then, on $[a_i, b_i - \frac 2n]$, we set
$$
\tilde g_n(x)
=\left\{\begin{array}{ll}
  g_n(a_i) & x\in [a_i, d],\\
  g_n(x) &\hbox{otherwise}.
 \end{array} \right.
$$
It is clear that $\tilde g_n$ has the desired properties.

Other cases are handled in a similar manner.
\end{proof}

We now show a comparison principle for solutions to the LG problem \eqref{lg} for continuous data.

\begin{proposition}\label{comparison}
Let $\Omega$ be convex and $f_1$, and $f_2$ continuous in $\partial \Omega$ satisfying the admissibility conditions \ref{admC1} and \ref{admC2} and such that $f_1\leq f_2$. Let $u_1$ and $u_2$ be the corresponding solutions to \eqref{lg} constructed in Section 3, then $u_1\leq u_2$.
\end{proposition}

\begin{proof}
Let $\Omega_n$ be the strict convex sets constructed in Lemma \ref{prap} and $f_1^n$ and $f_2^n$  be as defined in \ref{defn} on $\Omega_n$ from $f_1$ and $f_2$.
We assume that $v_1^n$ (resp. $v_2^n$) is the unique solution to \eqref{lg} on $\Omega_n$ with corresponding trace $f_1^n$ (resp. $f_2^n$) and $u_1^n$ (resp. $u_2^n$) its restriction to $\Omega$. By definition, we have $f_1^n\leq f_2^n$, then by \cite{sternberg}, we have, $v_1^n\leq v_2^n$.
We know by the proof of Theorem \ref{main1} that $u_1^n$ and $u_2^n$ converge pointwise correspondingly to $u_1$ and $u_2$ in $\Omega$, therefore 
 $u_1\leq u_2$.
 \end{proof}

Our goal is to show the following theorem. 
 \begin{theorem}\label{Main2}
Let us suppose that the geometric assumptions on $\Omega$, specified in Theorem \ref{main1} hold, and in particular $\Omega$ is convex. Moreover,
$f\in BV(\d\Omega)$
%\cap C(\d\Omega\setminus \bigcup_{k\in\cK} \ell_k)$, where $\ell_k$, $k\in \cK$,  are flat parts of $\d\Omega$. We assume that 
$f$ satisfies the admissibility conditions \eqref{admC1} and \eqref{admC2} and $f$ has finitely many humps. Then, there exists a solution $u$ to the least gradient problem (\ref{lg}).
\end{theorem}
\begin{proof}
Define $C_f=\{x\in \d\Omega, \text{$f$ is continuous at $x$}\}$. By Lemmas \ref{Aprox2}, \ref{gnadm1} and \ref{gnadm2} we construct a decreasing sequence of continuous functions $h_n$ such that 
$h_n\to f$ in $C_f$. For each $h_n$, by Theorem \ref{main1}, there exists a continuous solution $u_n$ to \eqref{lg} on $\Omega$, with trace $h_n$. By the comparison 
principle in Proposition \ref{comparison}, we have that $u_n$ is a decreasing sequence. Then, these functions converge to a function $u$ at every point $x\in \Omega$. By 
\cite{miranda}, $u$ is a least gradient function. 

By a similar token, by Lemmas \ref{Aprox2}, \ref{gnadm1} and \ref{gnadm2} we construct an increasing sequence of continuous functions $g_n$ such that 
$g_n\to f$ in $C_f$. For each $h_n$, by Theorem \ref{main1}, there exists a continuous solution $v_n$ to \eqref{lg} on $\Omega$ with trace $g_n$. By the comparison 
principle, in Proposition \ref{comparison}, we have that $v_n$ is a decreasing sequence, converging to a function $v$ at every point $x\in \Omega$. By 
\cite{miranda}, $v$ is a least gradient function.

We shall prove that 
$$
Tu(x) = f(x) = T v(x)\qquad\hbox{for } x\in C_f.
$$
We claim $f(x)\ge Tu(x)$ for $x\in C_f$.
If it were otherwise, $Tu(x)=t>f(x)=\tau$, then there would exist $s$ in $(\tau,t)$.
Since $h_n(x)\to f(x)$ then, there exists $N$ such that $h_N(x)<s$. By continuity of $u_N$, we gather that $u_N(y)<s$ for every $y$ in a neighborhood of $x$ in $\Omega$, we may assume that this is a ball $B(x,r).$
Since $u_n$ is a decreasing sequence, then $u_n(y)<s$ for all $n\geq N$. Letting $n\to \infty$, we get that $u(y)\leq s$ for all $y\in B(x,r)\cap \Omega$, contradicting the fact that $Tu(x)=t<s$. 

Now, we consider  sequence $g_n$  converging to $f$ from below and the corresponding solutions $v_n$ to the least gradient problem. The sequence $v_n$ converges to a least gradient function $v$.

The same argument as above implies that
$ Tv(x) \ge f(x)$.
Since we automatically have that 
$$
u_n(x) \ge u(x) \ge v(x) \ge v_n$$
we deduce from the above inequalities that
$$
f(x) = Tu(x) = Tv(x) = f(x)\qquad\hbox{for }x\in C_f.
$$
Since $C_f\subset \Omega$ has full measure, we deduce that $T u = f$, as desired. The same argument yields $T v = f$.
\end{proof}

\begin{remark}
 In the course of the above proof, we constructed two solutions $u$ and $v$ to the least gradient problem having the trace at the boundary. However, they need not be equal.
\end{remark}
 
 \subsection{Discontinuous data with infinitely many humps}
 We assume in this section that the data can have infinitely many humps on flat parts of $\p \Omega$. To be specific, we assume $\ell=[p_l,p_r]$ is a line segment, where $f$ has infinitely many humps $I_i$, $i\in\cI$, then since $f$ satisfies the admissibility condition \#2 the lengths of $I_i$ must converge to 0 and the humps endpoints must converge to the one of the endpoints of $\ell$. For the sake of simplicity of the presentation, we assume that $p_l$ is the point of accumulation of $I_i$'s endpoints.
 
 \begin{proposition} If $f\in BV(\d\Omega)$, $\ell=[p_l,p_r]$ is a flat part containing infinitely many humps $I_i$, $i \in \cI$, $p_l$ is the point where the humps accumulate, then
 $f$ is continuous at $p_l$. 
 \end{proposition}
 
 \begin{proof}
Since $f$ is in $BV$, we will take the so-called `good representative' of $f$, see  \cite[Theorem 3.28]{AFP00}. From now on, we assume that we work with such $f$.

We shall show that 
 $$
\lim_{x\to p_l} f|_{\ell}(x),\quad
\lim_{x\to p_l} f|_{\p\Omega\setminus\ell}(x)\text{\,\, exist}
 $$
and they are equal.
In fact, since $f$ is of bounded variation on $\ell$, then for every $\varepsilon>0$
$$ 
\sup_{|x_{i+1}-x_{i}|<\varepsilon} \sum_{i=0}^{n-1} |f(x_{i+1})-f(x_i)|<\infty,
$$
where $x_0,\ldots x_n$ is a partition of  $\ell$. 
We know that $f$ is constant on $(a_i,b_i)$ and jumps may occur at $a_i$'s and $a_i\to p_l$. Now, we take a sequence $x_n$ such that $|x_n - p_l|$ decreases to zero, then
$$
\left|f(x_n)-f(x_m)\right|=\left|f(x_n)-f(x_{n+1})+\cdots+ f(x_{m-1})-f({x_m})\right|\leq \sum_{i=1}^\infty \left|f(x_{i+1})-f(x_{i})\right| .
$$
Due to the boundedness of the total variation, the series $\sum_{n=0}^\infty |f(x_n)-f(x_{n+1})|$ converges. Hence, $f(x_n)$ is a Cauchy sequence and 
$$
\lim_{n\to \infty} f(x_n) = \lim_{x\to p_l} f|_{\ell}(x) =: f(p_l^+).
$$ %we will call this number $f(p_l^+)$.
The same argument yields 
$$
\lim_{x\to p_l} f|_{\p\Omega\setminus\ell}(x) =: f(p_l^-).
$$

In order to proceed, we need points $y_i$, $z_i$, which are minimizers of the left-hand-side of  the admissibility condition (\ref{DA}). If $f$ were continuous, then existence of $y_i$, $z_i$ such that
$$
\dist(a_i, f^{-1}(e_i) \cap (\partial\Omega \setminus I_i)) = |a_i - y_i|, \qquad
\dist(b_i, f^{-1}(e_i) \cap (\partial\Omega \setminus I_i)) = |b_i - z_i|
$$
would be obvious. When $f$ is not necessarily continuous, we proceed differently. We take
$$
\eta_i^k \in f^{-1}(e_i) \cap (\partial\Omega \setminus I_i), \qquad
\xi_i^k \in f^{-1}(e_i) \cap (\partial\Omega \setminus I_i)
$$
such that
$$
|a_i - \eta_i^k| \to \dist(a_i, f^{-1}(e_i) \cap (\partial\Omega \setminus I_i)) , \qquad 
|b_i - \xi_i^k| \to \dist(b_i, f^{-1}(e_i) \cap (\partial\Omega \setminus I_i).
$$
Since $\eta_i^k$ and $\xi_i^k$ are bounded, we may assume that 
$$
\lim_{k\to\infty} \eta_i^k =  y_i,  \qquad \lim_{k\to\infty} \xi_i^k =  z_i.
$$
Of course the limits $\lim_{k\to\infty}f( \eta_i^k) = e_i = \lim_{k\to\infty}f( \xi_i^k).$ Thus, abusing the notation, we write  $f(a_i) = f(y_i)$, $f(b_i) = f(z_i)$.

Now, since for any sequence $x_i\in (a_i,b_i)\subset\ell$ converging to $p_l$, there is a sequence $y_i\in \d\Omega \setminus\ell$ converging to $p_l$ and such that $f(x_i) = f(y_i)$, we infer that the one-sided limits agree, $f(p_l^-) = f(p_l^+)$. A good representative must be continuous at a point if the one-sided limits are equal. 
As a result,
%by the admissibility condition, for every $a_i$ there exists $c_i$ in $\d\Omega\setminus\ell$ %the juxtaposing side of edge $a$ 
%such that $f(a_i)=f(c_i)$ and $a_i-c_i\to 0$ as $i\to \infty$ by the admissibility condition and therefore similarly $f(c_i)\to f(0-)$ but since $f(c_i)=f(a_i)$ then $f(0+)=f(0-)$, and 
$f$ is continuous at $0$.
\end{proof}

When we deal with an infinite number of humps, then we assume that only one flat part contains an infinite number of them. This restriction is introduced solely for the sake of simplicity of the exposition.
%If the number of flat parts is infinite, then the argument can be  generalized by performing the same construction near every edge.

We construct the following sequence of least gradient functions on $\Omega$. 

Since $f$ might be discontinuous at $a_i$ or $b_i$, then we choose $x^{\star}_i\in [a_i,b_i]$ and $y^{\star}_i\in [y_i,z_i]$ such that $f(x^{\star}_i)=f(y^{\star}_i)$. Let $T_i$ be the triangle $\conv (a,x^{\star}_i,y^{\star}_i)$, and $Q_i$ the trapezoid\break\hfill $\conv (x^{\star}_{i+1,}x^{\star}_{i},y^{\star}_{i},y^{\star}_{i+1})$. We construct the following sequence of least gradient functions on $\Omega$ as follows.

Let $Q_0=\Omega\setminus T_1$ and 
 $$
 f_0(x)=\begin{cases} f(x) \qquad &x \in \partial Q_0\cap \partial \Omega,\\
 f(x^\star_1) \qquad &x\in [x^{\star}_1, y^{\star}_1]
 \end{cases}
 $$
and $v_0$ be a least gradient function on $Q_0$ with trace $f_0$ and we define,
$$
 u_0(z)=\begin{cases} v_0(z) \qquad &z \in Q_0,\\
 f(x^\star_1) \qquad &z\in T_1.
 \end{cases}
 $$ 
 \begin{lemma}\label{lm:u0 lg} Let us assume that $u_0$ is defined by the above formula.
Then, $u_0$ is a least gradient function in $\Omega$.
 \end{lemma}
 
 \begin{proof}
Let $u$ be a least gradient function on $\Omega$, then $\conv(a_1, b_1, d_1, c_1)$ is a fat level set of $u$. Hence, the restriction of $u$ on $Q_0$ has trace on $[x^{\star}_1, y^{\star}_1]$, and then $T_{Q_0}u=f_0$. Therefore, since $v_0$ is a least gradient function on $Q_0$ with trace $f_0$, we get that
$$|Du|(Q_0)\geq |Dv_0|(Q_0)=|Du_0|(Q_0).$$
We notice that the one-sided limits of $u$ on the segment $[x^{\star}_1,y^{\star}_1]$ are equal. 

The same is true about $u_0$.
%We also can see that the one-sided limits of $u_0$   on the segment $[x^{\star}_1, y^{\star}_1]$ are equal.
Since $u_0$ is constant in $T_1$, then 
$$|Du_0|(T_1)=0\leq |Du|(T_1).$$
As a result, we conclude that $|Du_0|(\Omega)\leq |Du|(\Omega)$. In addition, $T u_0=T u$ and $u$ %due to \cite[Proposition 4.1]{grs} 
is a least gradient function. Therefore,
$|Du_0|(\Omega)=|Du|(\Omega)$ and $u_0$ is a least gradient function too.
 \end{proof}
 
We now define the function $f_1$ on $Q_1$ as follows,
  $$
 f_1(x)=\begin{cases} f(x), \qquad &x \in \partial Q_1\cap \partial \Omega,\\
 f(x^\star_1), \qquad &x\in [x^{\star}_1,y^{\star}_1],\\
 f(x^{\star}_2), \qquad &x\in [x^{\star}_2,y^{\star}_2].
 \end{cases}
 $$
 Let $v_1$ be a least gradient function on $Q_1$ with trace $f_1$ and define the function
 $$
 u_1= v_0 \chi_{ Q_0} +  v_1\chi_{Q_1} + 
 f(x^\star_2)\chi_{T_2}.
 $$
 
Notice that $u_1=u_0$ on $Q_0$ and
 $$
 T u_1=\begin{cases} f(x), \qquad & x \in \partial (Q_0\cup Q_1)\cap \partial \Omega,\\
 f(x^\star_2), \qquad & x\in \partial T_2.
 \end{cases}
 $$ 
 
Since $\conv(a_1,b_1,d_1,c_1)$ and $\conv(a_2,b_2,d_2,c_2)$ are fat level sets of a least gradient functions on $\Omega$ with trace equal to $T u_1$, we infer as in Lemma 
\ref{lm:u0 lg} that $u_1$ is a least gradient function.
 
Recursively, we construct  the sequence of boundary data
 $$
 f_n(x)=\begin{cases} f(x), \qquad &x \in \partial Q_n\cap \partial \Omega,\\
 f(x^\star_n), \qquad &x\in [x^{\star}_n, y^{\star}_n],\\
 f(x^\star_{n+1}), \qquad &x\in [x^{\star}_{n+1}, y^{\star}_{n+1}].
 \end{cases}
 $$
We take $v_n$ to be a least gradient function on $Q_n$ with trace $f_n$ and define the function, 
$$
u_n = \sum_{k=0}^n v_k \chi_{Q_k} +
f(x^\star_{n+1}) \chi_{T_{n+1}}.
$$
Arguing as before, we come to the conclusion that $u_n$ is a least gradient function and 
$$
 T u_n=\begin{cases} f(x), \qquad &x \in \partial (Q_0\cup Q_1\cup \cdots\cup Q_n)\cap \partial \Omega,\\
 f(x^\star_{n+1}), \qquad &x\in \partial T_{n+1}.
\end{cases}
$$
Moreover, we notice  that
 $u_n(z)=v_i(z)=u_i(z)$ for $z\in Q_i$, $i=1,\cdots,n-1$ and $u_n(z)=  f(x^\star_{n+1})$ in $T_{n+1}$. 
 
By \cite{miranda}, the sequence $\{ u_n\}_{n=1}^\infty$ converges, up to a subsequence, to a least gradient function $u$. The construction of $u_n$ is such that $u_n = u_k$ on $\bigcup_{i=1}^k Q_i$ for $n\ge k$. As a result $T_\Omega$ and $T_{\bigcup_{i=1}^k Q_i}$ are equal on 
$\p (\bigcup_{i=1}^k Q_i) \setminus\Omega$. From this fact we deduce $T_\Omega u = f$.
Hence, we just have showed the following theorem:
 
\begin{theorem}\label{Main3}
Let us suppose that the geometric assumptions on $\Omega$ specified in Theorem \ref{main1} hold, in particular $\Omega$ is convex. Moreover,
$f\in BV(\d\Omega)$
%\cap C(\d\Omega\setminus \bigcup_{k\in\cK} \ell_k)$, where $\ell_k$, $k\in \cK$,  are flat parts of $\d\Omega$.  We assume that $f$ 
satisfies the admissibility conditions \eqref{admC1} and \eqref{admC2}. Then, there exists a solution $u$ to the least gradient problem (\ref{lg}).
 \end{theorem}

\section{Examples}\label{examples}
We present a few examples showing how our theory works. We set $\Omega = (-L,L)\times (-1,1)$, where $L>1$ and
$g:(-L,L)\to \bR$, by formula $g(x) = L^2 -x^2$. Furthermore, we define
$f:\d\Omega\to \bR$, by formula $f(x_1, \pm 1) = g(x_1)$ and $f(\pm L, x_2) =0$. We take $\lambda>0$. 

Here is the first example.
We introduce $f_\lambda = \min\{ f, g(L-\lambda)\}.$
\begin{cor}\label{co-1}
 If $\Omega$ and $f_\lambda$ are defined above, then:\\
 a) If $\lambda\in (0,1)$, then the admissibility condition \# 2 holds and $u$, a solution to (\ref{lg}), is given by the following formula,
 $$
 u_\lambda(x_1,x_2) = f_\lambda(x_1).
 $$\\
 b) If $\lambda\in (1,L)$, then the admissibility condition \# 2 
 is violated.\\
 c) If $\lambda=1$, then $u$ given below is  a solution to (\ref{lg}),
 $$
 u(x_1,x_2) = f_1(x_1).
 $$
\end{cor}
\begin{proof}
We take care of part a).
The admissibility condition \# 2 is easy to check. The formula for
$u_\lambda$ is easy to find after discovering solutions in
$\Omega_n$. Finally, we notice that $u$ is a pointwise limit of
$u_\lambda$ as $\lambda$ goes to 1. We use here the fact that an $L^1$
limit of least gradient functions is of least gradient. Moreover, it
is obvious that the limit has the right trace. 

The rest may be established in a similar way.
\end{proof}

Now, we consider a case of discontinuous data. We set $h(x) = x+L$ for $|x|<L$.  For $\mu\in\bR$ we define,
$$
v_\mu (x_1,x_2) =
\left\{
\begin{array}{ll}
h(x_1) & x_1 \in (-L,L), \ x_2 =1,\\
0 & x_1 \in (-L,L), \ x_2 =0 \hbox{ or } x_1 =-L, \ x_2\in(-1,1),\\
\mu & x_1 =1, \ x_2\in (-1,1).
\end{array}
\right.
$$
We notice:
\begin{cor}\label{co-2}
 Let us suppose $\Omega$ and  $v_\mu\in BV(\d\Omega)$ are defined above. Then,\\
(a) for $\mu > 2L$ condition \#2 is violated;\\
(b) for $\mu = 2L$ conditions \#1 and  \#2 hold;\\
(c) for $\mu < 2L$ condition \#1  is violated.
\end{cor}

\begin{proof}
Part (a) and (c) are easy to see.
If $\mu = 2L$, then it is easy to check
the admissibility conditions \#1 and \# 2. The positions of sets $\d\{ u\ge t\}$ follow from solutions of the approximate problem on $\Omega_n$.

\end{proof}

\begin{proposition}
In all the cases above, when the admissibility conditions are
violated, i.e. in Corollary \ref{co-1} (b),  Corollary \ref{co-2} (a),
(c) there is no solution to the least gradient problem.
\end{proposition}
\begin{proof}
Let us suppose otherwise. Then, we notice that for almost every
$x_2\in (-1,1)$ the function $u(\cdot, x_2)$ belongs to
$BV(-L,L)$. Indeed,
$$
\int_{-1}^1\int_{-L}^L | D_{x_1} u(x_1, x_2)| \,dx_2 \le 
\int_{-1}^1\int_{-L}^L | D u(x_1, x_2)| \le M.
$$
In particular  $u(\cdot, x_2)$ has a trace at $x_1 =L$ for
a.e. $x_2\in (-1,1)$. In other words, for any $x_n$ increasing sequence
converging to $L$, we have 
$$
\lim_{n\to\infty} u(x_n, x_2) = \mu.
$$
We set $t_n = u(x_n, x_2)$ and consider 
$$
E_{t_n} = \{ (x_1, x_2) :\  u(x_1, x_2)\ge t_n\}
$$
for values $t_n\in (2L, \mu)$. Since $\partial E_{t_n}$ is a minimal
surface, then it must be a line segment, which must intersect
$\partial\Omega$. %\marginpar{details?}
We set $\{a_n, b_n\} = \partial E_{t_n} \cap  \partial \Omega.$
Since $v_\mu$ does not attain any values in the interval $(2L, \mu)$,
we deduce that $a_n, b_n$ must belong to $\partial\Omega \setminus
\{(L, \pm 1)\}$. But this implies that all $\partial E_{t_n}$ are
contained in $(L, x_2): \ |x_2| \le 1\}$, i.e. in the boundary of
$\Omega$, contrary to the assumptions that $(x_n,x_2)\in \Omega$. We
reached a contradiction. Our claim follows.
\end{proof}

Finally, we construct a region $\Omega$ and a continuous function on its boundary with infinitely many humps. We define $\Omega$ to be bounded by the following curves, $\ell_1 = [0, L_1]\times \{0\}$, $\ell_2$ is a line segment of length $L_1$ forming and angle $\alpha$ at the origin. Moreover, $\alpha(0, \frac\pi2)$. The third arc, $\cC$,  is a part of a half-circle with radius 
$r= \sqrt 2 L_1 \sqrt{1-\cos \alpha}$. 

We set 
$$
L_{2k+1} = L_1 \prod_{i=1}^k \left( 
\frac{(1-\sin \alpha)^2}{1 + \sin \alpha} - \varepsilon_i \right),
\qquad L_{2k} = L_{2k-1} \frac{ 1-\sin \alpha}{1 + \sin \alpha},\quad k\ge 1,
$$
where $0<\varepsilon_i$ is decreasing to zero and $\varepsilon_1 < \frac 12 \frac{(1-\sin \alpha)^2}{1 + \sin \alpha}$.

We set the position of the hump by defining $a_k := L_{2k}$, $b_k = L_{2k-1}$. We define $f$ on $\ell_1$ by setting $f(x) = \frac{(-1)}{k}^{k+1}$ for $x\in (a_k, b_k)$, $k\ge1$. We extend $f$ to  $\ell_1\setminus \bigcup_{k=1}^\infty (a_k, b_k)$ by linear functions.

Let $\pi$ denote the orthogonal projection onto the line containing  $\ell_2$. We set 
$$
a_k' := \pi (a_k,0), \qquad b_k' := \pi ( b_k,0).
$$
We set  $f(x) = \frac{(-1)^{k+1}}{k}$ for $x\in (a_k', b_k')$, $k\ge1$. We extend $f$ to  $\ell_2\setminus \bigcup_{k=1}^\infty (a_k, b_k)$ by linear functions. We set $f$ on $\cC$ to be equal to 1.

It is easy to check see that we have just proved the following fact:
\begin{proposition}
Let us suppose that $\Omega$ is given above. Then, function $f$ constructed above is continuous on $\partial \Omega$ and it satisfies the admissibility condition \#2. \qed
\end{proposition}

\section*{Acknowledgement} The work of the authors was in part
supported by the Research Grant 2015/19/P/ST1/02618 financed by the
National Science Centre, Poland, entitled: Variational Problems in
Optical Engineering and Free Material Design.

PR was in part supported by  the Research Grant no 2013/11/B/ST8/04436 financed by the National
Science Centre, Poland, entitled: Topology optimization of engineering structures. An approach synthesizing the methods of:
free material design, composite design and Michell-like trusses.

The authors also thank Professor Tomasz Lewi\'nski of Warsaw Technological University for stimulating discussions and constant encouragement.

%All the authors thank the referee for his/her insightful comments, which helped to improve the paper.

\vspace{-.2cm}
\hspace{3cm}
\begin{wrapfigure}{l}{0.15\textwidth}
\includegraphics[width=2.5 cm, height= 2cm]{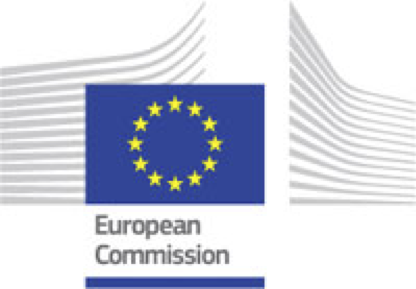}  
  %\caption{Birds}
\end{wrapfigure}\\
\\{\footnotesize This project has received funding from the European Union's Horizon 2020 research and innovation program under the Marie Curie grant agreement No 665778.}
\newline

\end{document}